\documentclass[10pt,twoside, a4paper, english, reqno]{amsart}
\usepackage{amssymb,amsfonts,latexsym,color,tikz,float,pgf,subfig,caption,bm}
\usepackage{graphics,verbatim}
\usepackage{graphicx}
\usepackage{a4wide}
\usepackage{hyperref}
\usepackage{enumerate}
\usepackage[utf8]{inputenc}
\usepackage{color}
\usepackage[normalem]{ulem}

\newtheorem{theorem}{Theorem}[section]
\newtheorem{proposition}[theorem]{Proposition}
\newtheorem{lemma}[theorem]{Lemma}
\newtheorem{corollary}[theorem]{Corollary}
\theoremstyle{remark}
\newtheorem{remark}[theorem]{Remark}
\theoremstyle{definition}

\DeclareMathOperator{\divergence}{div}

\newcommand{\eps}{\varepsilon}
\renewcommand{\epsilon}{\varepsilon}

\newcommand{\abs}[1]{\left\vert#1\right\vert}

\newcommand{\norm}[1]{\left\Vert#1\right\Vert}

\newcommand{\be} {\begin{equation}}
\newcommand{\ee} {\end{equation}}
\newcommand{\bea} {\begin{eqnarray}}
\newcommand{\eea} {\end{eqnarray}}
\newcommand{\Bea} {\begin{eqnarray*}}
\newcommand{\Eea} {\end{eqnarray*}}
\newcommand{\pa} {\partial}

\newcommand{\al} {\alpha}

\newcommand{\de} {\delta}
\newcommand{\na}{\nabla}
\newcommand{\ga} {\gamma}

\newcommand{\De} {\Delta}
\newcommand{\la} {\lambda}

\newcommand{\nequiv} {\not\equiv}

\newcommand{\f}{\frac}
\newcommand{\R}{\mathbb R}
\newcommand{\N}{\mathbb N}
\newcommand{\Rn}{\mathbb R^N}

\renewcommand{\S}{\mathbb{S}}
\newcommand{\deb}{\rightharpoonup}
\catcode`\@=11

\newenvironment{bvp}{\left\{\begin{aligned}  }{\end{aligned}\right.}

\newcommand{\Ds}{\mathcal{D}^{s,2}(\R^N)} 
\newcommand{\Dext}{\mathcal{D}^{1,2}(\R^{N+1}_+;t^{1-2s})} 

\newcommand{\dx}{\,\mathrm{d}x} 
\newcommand{\dy}{\,\mathrm{d}y} 
\newcommand{\dS}{\,\mathrm{d}S} 
\newcommand{\ds}{\,\mathrm{d}S'} 
\newcommand{\dxdt}{\, \mathrm{d}t \,\mathrm{d}x} 
\newcommand{\dxdy}{\, \mathrm{d}x \,\mathrm{d}y} 

\DeclareMathOperator{\Tr}{\mathrm{Tr}} 
\DeclareMathOperator{\supp}{\mathrm{supp}} 

\numberwithin{equation}{section}

\title{On fractional multi-singular Schr\"odinger operators:
  positivity and localization of binding}
\author{Veronica Felli}
\address{Veronica Felli 
\newline \indent Dipartimento di Scienza dei Materiali, Università
degli Studi di Milano-Bicocca,
\newline \indent Via Cozzi 55, 20125 Milano, Italy.}
\email{veronica.felli@unimib.it}
\author{Debangana Mukherjee}
\address{Debangana Mukherjee
\newline \indent Department of Mathematics and Statistics, Masaryk University,
\newline \indent Kotlářská 267/2, 611 37 Brno, Czech Republic.}
\email{mukherjeed@math.muni.cz}
\author{Roberto Ognibene}
\address{Roberto Ognibene
\newline \indent Dipartimento di Matematica e Applicazioni, Università
degli Studi di Milano-Bicocca,
\newline \indent Via Cozzi 55, 20125 Milano, Italy.}
\email{roberto.ognibene@unimib.it}

\begin{document}

\maketitle

\begin{abstract}
  In this work we investigate positivity properties of nonlocal
  Schr\"odinger type operators, driven by the fractional
    Laplacian, with multipolar, critical, and  locally homogeneous potentials. On one
  hand, we develop a criterion that links the positivity of the
  spectrum of such operators with the existence of certain positive
  supersolutions, while, on the other hand, we study the localization
  of binding for this kind of potentials. Combining these two tools
  and performing an inductive procedure on the number of poles, we
  establish necessary and sufficient conditions for the existence of a
  configuration of poles that ensures the positivity of the
  corresponding Schr\"odinger operator.
\end{abstract}



\section{Introduction}\label{sec:intr}
Let $s\in(0,1)$ and $N>2s$. Let us consider $k\geq 1$ real numbers $\lambda_1,\dots,\lambda_k$ (sometimes called \emph{masses}) and $k$ \emph{poles} $a_1,\dots,a_k\in \R^N$ such that $a_i\neq a_j$ for all $i,j=1,\dots,k,~i\neq j$. The main object of our investigation is the operator
\begin{equation}\label{eq:frac_oper}
\mathcal{L}_{\la_1,\dots ,\la_k,a_1,\dots ,a_k}:=(-\De)^s-\sum_{i=1}^{k}\frac{\la_i}{|x-a_i|^{2s}}\qquad\text{in }\R^N.
\end{equation}
Here $(-\Delta)^s$ denotes the fractional Laplace operator, which acts
on functions $\varphi\in C_c^\infty(\R^N)$ as 
 \begin{equation*}
   (-\Delta)^s\varphi (x):=
   C(N,s)\,\mathrm{P.V.}\,\int_{\R^N}
   \frac{\varphi(x)-\varphi(y)}{\abs{x-y}^{N+2s}}\dy
   =C(N,s)\lim_{\rho\to 0^+}
\int_{\abs{x-y}>\rho}\frac{\varphi(x)-\varphi(y)}{\abs{x-y}^{N+2s}}\dy,
\end{equation*}
 where $\mathrm{P.V.}$ means that the integral has to be seen in the principal value sense and
 \[
  C(N,s)=\pi^{-\frac{N}{2}}2^{2s}\frac{\Gamma\left(\frac{N+2s}{2}\right)}{\Gamma(2-s)}s(1-s),
 \]
 with $\Gamma$ denoting the usual Euler's Gamma function. Hereafter, we refer to an operator of the type $(-\Delta)^s-V$ as a
 \emph{fractional Schr\"odinger operator} with potential $V$.

 One of the reasons of  mathematical interest in operators of type \eqref{eq:frac_oper}
lies in  the criticality of potentials of order $-2s$, which have the
same scaling rate as the $s$-fractional Laplacian.

We introduce, on $C_c^\infty(\R^N)$, the following positive definite bilinear form, associated to $(-\Delta)^s$
\begin{equation}\label{eq:def_scalar}
 (u,v)_{\Ds}:=\frac{1}{2}C(N,s)\int_{\R^{2N}}\frac{(u(x)-u(y))(v(x)-v(y))}{\abs{x-y}^{N+2s}}\dxdy 
\end{equation}
and we define the space $\Ds$ as the completion of $C_c^\infty(\R^N)$
with respect to the norm $\norm{\cdot}_{\Ds}$ induced by the scalar
product \eqref{eq:def_scalar}. Moreover, 
the following quadratic form is naturally associated to the operator
$\mathcal{L}_{\lambda_1,\dots,\lambda_k,a_1,\dots,a_k}$
\begin{equation}\label{eq:quadr_form}
\begin{aligned}
   Q_{\la_1,\dots ,\la_k,a_1,\dots ,a_k}(u):&=\frac{1}{2}C(N,s)\int_{\R^{2N}}\frac{\abs{u(x)-u(y)}^2}{\abs{x-y}^{N+2s}}\dxdy-\sum_{i=1}^{k}\la_i \int_{\Rn}\frac{\abs{u(x)}^2}{|x-a_i|^{2s}}\dx \\
   &=\norm{u}_{\Ds}^2-\sum_{i=1}^{k}\la_i \int_{\Rn}\frac{\abs{u(x)}^2}{|x-a_i|^{2s}}\dx.
\end{aligned}
\end{equation}
We observe that 
$Q_{\la_1,\dots ,\la_k,a_1,\dots ,a_k}$ is well-defined 
 on
$\Ds$ thanks to the validity of the following fractional Hardy inequality
proved in \cite{Herbst1977}:
\begin{equation}\label{eq:hardy}
 \gamma_H \int_{\R^N}\frac{\abs{u(x)}^2}{\abs{x}^{2s}}\dx\leq \norm{u}_{\Ds}^2 \quad\text{for all }u\in\Ds,
\end{equation}
where the constant 
\begin{equation*}
\gamma_H  =\gamma_H(N,s):= 2^{2s}\frac{\Gamma^2\left(\frac{N+2s}{4}\right)}{\Gamma^2\left(\frac{N-2s}{4}\right)}
\end{equation*}
is optimal and not attained.

One goal of the present paper is to find necessary and sufficient
conditions (on the masses
$\la_1,\dots,\la_k$) for the existence of a configuration of poles
$(a_1,\dots,a_k)$ that guarantees the positivity of the quadratic form
\eqref{eq:quadr_form}, extending to the fractional case some results
obtained in \cite{Felli2007} for the classical Laplacian. The
quadratic form $Q_{\la_1,\dots ,\la_k,a_1,\dots
  ,a_k}$ is said to be \emph{positive definite} if
\begin{equation*}
\inf_{u\in\Ds\setminus \{0\}}\frac{ Q_{\la_1,\dots ,\la_k,a_1,\dots,a_k}(u)}{\norm{u}_{\Ds}^2}>0.
\end{equation*}
In the case of a single pole (i.e.
$k=1$), the fractional Hardy inequality \eqref{eq:hardy} immediately
answers the question of positivity: the quadratic form $Q_{\lambda,a}$
is positive definite if and only if
$\lambda<\ga_H$. Hence our interest in multipolar potentials is
justified by the fact that the location of the poles (in particular
the shape of the configuration) could play some role in the positivity
of \eqref{eq:quadr_form}. Furthermore, one could expect that some
other conditions on the masses may arise when
$k>1$. We mention that several  authors have approached the problem of multipolar
singular potentials, both for the classical Laplacian,  see e.g. \cite{berchio,bosi-esteban-dolbeault,zuazua,Almeida2017,faraci,Ferreira2013}
and for the fractional case, see  \cite{Ferreira2017}.

A fundamental tool in our arguments is the well known
Caffarelli-Silvestre extension for functions in $\Ds$, which allows us
to
study the nonlocal operator $(-\Delta)^s$ by means of a boundary value
problem driven by a local operator in
$\R^{N+1}_+:=\{(t,x):t\in(0,+\infty), x\in \R^N\}$.  We introduce
 the space $\Dext$, defined as the completion of
$C_c^{\infty}(\overline{\R^{N+1}_+})$ with respect to the norm
 \[
 \norm{U}_{\Dext}:=\left( \int_{\R^{N+1}_+}t^{1-2s}\abs{\nabla U}^2\dxdt \right)^{1/2}.
 \]
 We have that there exists a well-defined and continuous trace map 
 \begin{equation}\label{eq:3}
 \Tr\colon \Dext \to \Ds
\end{equation}
 which is onto,  see, for instance, \cite{Brandle2013}. Let us now consider, for $u\in\Ds$, the following minimization problem
 \begin{equation}\label{eq:min_ext}
 \min\left\{ \int_{\R^{N+1}_+}t^{1-2s}\abs{\nabla \Phi}^2\dxdt\colon \Phi\in\Dext,~ \Tr \Phi=u \right\}.
 \end{equation}
 One can prove that there exists a unique function $U\in\Dext$
 (which we call the \emph{extension} of
 $u$) attaining \eqref{eq:min_ext}, i.e.
 \begin{equation}\label{eq:min_ext_ineq}
 \int_{\R^{N+1}_+}t^{1-2s}\abs{\nabla U}^2\dxdt \\ \leq \int_{\R^{N+1}_+}t^{1-2s}\abs{\nabla \Phi}^2\dxdt
\end{equation}
for all $\Phi \in \Dext$ such that $\Tr\Phi=u$.
 Furthermore, in \cite{Caffarelli2007} it has been proven that 
 \begin{equation}\label{eq:caff_silv}
 \int_{\R^{N+1}_+}t^{1-2s}\nabla U\cdot\nabla \Phi \dxdt=\kappa_s(u,\Tr\Phi)_{\Ds}\quad\text{for all }\Phi\in\Dext,
 \end{equation}
 where
 \begin{equation}\label{eq:kappa_s}
 \kappa_s:=\frac{\Gamma(1-s)}{2^{2s-1}\Gamma(s)}.
 \end{equation}
We observe that \eqref{eq:caff_silv} is the variational formulation of
the following problem
 \begin{equation}\label{eq:strong_ext}
 \left\{\begin{aligned}
 -\divergence(t^{1-2s}\nabla U)&=0, &&\text{in }\R^{N+1}_+, \\
 -\lim_{t\to 0}t^{1-2s}\frac{\partial U}{\partial t}&=\kappa_s(-\Delta)^s u, && \text{on }\R^N.
 \end{aligned}\right. 
 \end{equation}
 In the classical (local) case, the problem of  positivity of Schr\"odinger operators
 with multi-singular Hardy-type potentials was addressed in \cite{Felli2007}. In that
 article, the authors tackled the problem making use of a
 \emph{localization of binding} result that
 provides, under certain assumptions, the positivity of the sum of two
 positive operators, by translating one of them through a
 sufficiently long vector. This argument is based, in turn, on a criterion
 which relates the positivity of an operator to the existence of a
 positive supersolution, in the spirit of Allegretto-Piepenbrink
 Theory (see \cite{Allegretto1974,Piepenbrink1974}). As one can
 observe in \cite{Felli2007}, the strong suit of the local case is
 that the study of the action of the operator can be substantially
 reduced to neighbourhoods of the
 singularities. However, this is not possible in the fractional
 context due to nonlocal effects:
 in the present paper  we overcome this issue by taking into consideration
 the Caffarelli-Silvestre extension \eqref{eq:strong_ext}, which yields a local formulation of the
 problem. 

The equivalence between the fractional problem in $\R^N$ and the
Caffarelli-Silvestre extension problem in $\R^{N+1}_+$ allows us to
characterize the coercivity properties of quadratic forms on $\Ds$ 
in terms of quadratic forms on $\Dext$. 
We say that a function $V\in L^1_{\rm loc}(\R^N)$ satisfies the
\emph{form-bounded condition} if 
\begin{equation}
  \tag{$FB$}\label{eq:14}
  \sup_{\substack{u\in\Ds\\u\not\equiv0}}\frac{\int_{\R^N}|V(x)|u^2(x)\dx}{\|u\|_{\Ds}^2}<+\infty.
\end{equation}
Let $\mathcal H$ be the class of potentials satisfying the
form-bounded condition, i.e.
\[
\mathcal H=\{V\in L^1_{\rm loc}(\R^N):V\text{ satisfies
}\eqref{eq:14}\}.
\]
It is easy to understand that, if $V\in\mathcal H$, then $Vu\in
(\Ds)^\star$ for all $u\in\Ds$ and the quadratic form $u\mapsto
\|u\|_{\Ds}^2-\int_{\R^N}Vu^2$ is well defined in $\Ds$. For all
$V\in\mathcal H$ we define 
\begin{equation}\label{eq:def_inf}
\mu(V)= 
 \inf_{u \in \Ds\setminus\{0\}}\frac{ \norm{u}_{\Ds}^2-\int_{\R^N}V u^2\dx}{\displaystyle\norm{u}_{\Ds}^2 }  
\end{equation}
and observe that $\mu(V)>-\infty$.
\begin{lemma}\label{lemma:equiv_inf}
 Let $V\in\mathcal H$. 
 Then
 \begin{equation}\label{eq:def_mu}
 \mu(V)=\inf_{\substack{U \in \Dext \\ U\nequiv 0}}\frac{\displaystyle\int_{\R^{N+1}_+}t^{1-2s}|\na U|^2\dxdt-\kappa_s\int_{\Rn}V |\Tr U|^2\dx}{\displaystyle\int_{\R^{N+1}_+}t^{1-2s}|\na U|^2\dxdt}.
   \end{equation}
 \end{lemma} 
In the present paper we will focus our attention on 
the following class of potentials
\begin{multline*}
  \Theta:=\Bigg\{ V(x)=\sum_{i=1}^{k}\frac{\la_i
    \chi_{B'(a_i,r_i)}(x)}{|x-a_i|^{2s}}+\frac{\la_\infty \chi_{\Rn
      \setminus B'_R}(x)}{|x|^{2s}}+W(x)\colon r_i,R>0,~k\in\N, \\
  a_i\in\R^N,~a_i\neq a_j~\text{for }i\neq j,~
  \lambda_i,\lambda_\infty<\ga_H,~W\in L^{N/2s}(\R^N)\cap
  L^\infty(\R^N)\Bigg\},
\end{multline*}
where, for any $r>0$ and $x\in \R^N$, we denote
\[
 B'(x,r):=\{y\in\R^N\colon \abs{y-x}<r\}\quad\text{and}\quad B'_r:=B'(0,r).
\]
We observe that, when considering a potential  $V\in \Theta$, it is
not restrictive to assume that the
sets $B'(a_i,r_i)$ and $\Rn\setminus B'_R$ appearing in its
representation are mutually disjoint, up to redefining the remainder
$W$.

It is easy to see that, for instance,
\[
 \sum_{i=1}^k\frac{\lambda_i}{\abs{x-a_i}^{2s}}\in\Theta,\quad\text{when
 }
\lambda_i<\ga_H~\text{for all }i=1,\dots,k\text{ and }\sum_{i=1}^k\lambda_i<\ga_H.
\]
We observe that any $V \in \Theta$ satisfies the  form-bounded
condition, i.e. $\Theta\subset\mathcal H$, thanks to
the fractional Hardy and Sobolev inequalities stated in
\eqref{eq:hardy} and \eqref{eq:sobolev} respectively.

 Our first main result is a criterion that provides the equivalence
 between the positivity of $\mu(V)$ for potentials $V\in\Theta$
 and the existence of a positive supersolution to a certain (possibly
 perturbed) problem. This criterion is reminiscent
 of the Allegretto-Piepenbrink Theory, developed in 1974 in 
 \cite{Allegretto1974,Piepenbrink1974} (see also
 \cite{Agmon1983,Agmon1985,Moss1978,Pinchover2016}). As far as we
 know, the result contained in the following lemma is new in the
 nonlocal framework; nevertheless,
 some tools from the Allegretto-Piepenbrink Theory have been used in
 \cite{Frank2008a,Moroz2012} to prove some Hardy-type fractional inequalities.

\begin{lemma}[Positivity Criterion]\label{Criterion}
	Let $V=\sum_{i=1}^k \frac{\la_i
		\chi_{B'(a_i,r_i)}(x)}{|x-a_i|^{2s}}+\frac{\la_\infty \chi_{\R^{N}
			\setminus B'_R}(x)}{|x|^{2s}}+W(x) \in \Theta$ and let
	$\tilde{V}\in L^\infty_{\textup{loc}}(\R^N\setminus\{a_1,\dots,a_k\})$
	be such that $V\leq \tilde{V}\leq \abs{V}$ a.e. in $\R^N$. The following two assertions hold true.
	\begin{itemize}
		\item[(I)]  Assume that there exist some $\eps>0$ and
                  a function $\Phi
                  \in \Dext$ such that  $\Phi>0$ in $\overline{\R^{N+1}_+} \setminus
                \{(0,a_1), \dots, (0,a_k)\}$, $\Phi\in C^{0}\left(\overline{\R^{N+1}_+} \setminus
		\{(0,a_1),\dots, (0,a_k)\}\right)$,  and 
		\begin{equation}\label{eq:1}
		\int_{\R^{N+1}_+}t^{1-2s}\na \Phi\cdot \na U\dxdt \geq \kappa_s \int_{\R^{N}}
		(V+\eps \tilde{V})\Tr \Phi\Tr U\dx,
		\end{equation}
		for all $U \in \Dext,~ U\geq 0$ a.e. in $\R^{N+1}_+$. Then
		\begin{equation}\label{eq:criterion_inf}
			 \mu(V)\geq \epsilon/(\epsilon+1).
		\end{equation}
              \item[(II)] Conversely, assume that $\mu(V)>0$. Then
                there exist $\eps>0$ (not depending on $\tilde{V}$)
                and $\Phi~\!\!\in~\!\!\Dext$ such that $\Phi>0$ in
                $\R^{N+1}_+$,
                $\Phi\in C^{0}\left(\overline{\R^{N+1}_+} \setminus
                  \{(0,a_1),\dots, (0,a_k)\}\right)$,
                $\Phi\geq 0$ in
                $\overline{\R^{N+1}_+} \setminus \{(0,a_1), \dots,
                (0,a_k)\}$,
                and \eqref{eq:1} holds for every $U \in \Dext$
                satisfying $U\geq 0$ a.e. in $\R^{N+1}_+$.  If, in
                addition, we assume that $V$ and $\tilde{V}$ are
                locally H\"older continuous in
                $\R^N\setminus\{a_1,\dots,a_k\}$, then $\Phi>0$ in
                $\overline{\R^{N+1}_+} \setminus \{(0,a_1), \dots,
                (0,a_k)\}$.
	\end{itemize}
\end{lemma}

In order to use statement (I) to obtain positivity of a given
Schr\"odinger operator with potential in $\Theta$, it is crucial to
exhibit a weak supersolution to the corresponding Schr\"odinger  equation,
i.e. a function satisfying \eqref{eq:1}, which is \emph{strictly
  positive} outside the poles. Nevertheless, the application of
maximum  principles to prove positivity of solutions to
singular/degenerate extension
problems is more delicate than in the classic case, due to regularity issues (see the Hopf type
 principle proved in \cite[Proposition 4.11]{Cabre2014} and recalled
 in Proposition \ref{prop:cabre_sire} of the Appendix).  For this
 reason, in order to apply the above criterion in Sections
 \ref{sec:perturbation} and \ref{sec:localization}, we will develop an
 approximation argument introducing a class of more regular potentials
 (see \eqref{eq:def_Theta_s}).

The following theorem, whose proof heavily relies on Lemma
\ref{Criterion}, fits in the theory of 
\emph{Localization of Binding}, whose aim is study the lowest eigenvalue of Schr\"odinger operators of the type
\[
 -\Delta + V_1+V_2(\cdot-y),\quad y\in \R^N,
\]
in relation to the potentials $V_1$ and $V_2$ and to the translation
vector $y\in\R^N$. The case in which $V_1$ and $V_2$ belong to the
Kato class has been studied in \cite{Pinchover1995}, while Simon in
\cite{Simon1980} analyzed the case of compactly supported potentials;
singular inverse square potentials were instead considered in \cite{Felli2007}.
 Our result concerns the fractional case and provides sufficient conditions on the potentials and on the length of the translation for the positivity of the corresponding fractional Schr\"odinger operator.

\begin{theorem}[Localization of Binding]\label{thm:separation}
Let
\begin{gather*}
 V_1(x)=\sum_{i=1}^{k_1}\frac{\la_i^1\chi_{B'(a_i^1,r_i^1)}(x)}{|x-a_i^1|^{2s}}+\frac{\la_\infty^1 \chi_{\R^N \setminus B'_{R_1}}(x)}{|x|^{2s}}+W_1(x) \in \Theta, \\
 V_2(x)=\sum_{i=1}^{k_2}\frac{\la_i^2\chi_{B'(a_i^2,r_i^2)}(x)}{|x-a_i^2|^{2s}}+\frac{\la_\infty^2 \chi_{\R^N \setminus B'_{R_2}}(x)}{|x|^{2s}}+W_2(x) \in \Theta,
\end{gather*}
and assume $\mu(V_1),\mu(V_2)>0$ and $\la_\infty^1+\la_\infty^2<\ga_H$. Then there exists $R>0$ such that, for every $y \in \Rn\setminus \overline{B'_R}$,
\[
 \mu(V_1(\cdot)+V_2(\cdot-y))>0.
\]
\end{theorem}
Combining the previous theorem with an inductive procedure on the
number of poles $k$, we obtain a necessary and sufficient condition
for positivity of the operator \eqref{eq:frac_oper}.

\begin{theorem}\label{theorem}
Let $(\la_1, \dots \la_k) \in \R^k.$ Then 
\begin{equation}\label{lambda-i}
\la_i<\ga_H \quad\text{for all }i=1,\dots, k, \quad\text{and}\quad \sum_{i=1}^{k}\la_i<\ga_H
\end{equation}
 is a necessary and sufficient condition for the existence of a configuration of poles $(a_1,\dots,a_k)$ such that the quadratic form $Q_{\la_1,\dots,\la_k,a_1,\dots,a_k}$ associated to the operator $\mathcal{L}_{\la_1,\dots,\la_k,a_1,\dots,a_k}$ is positive definite. 
\end{theorem}

Besides the interest in the existence of a configuration of poles
making $Q_{\lambda_1,\dots,\lambda_k,a_1,\dots,a_k}$ positive
definite, one can search for a condition on the masses
$\lambda_1,\dots,\lambda_k$ that guarantees the positivity of this
quadratic form for every configuration of poles; in this direction, an
answer is given by the following theorem (we refer to
\cite[Proposition 1.2]{FT} for an analogous
result in the classical case of the Laplacian with multipolar inverse square potentials).

\begin{theorem}\label{Prop-2}
Let $t^+:=\max\{0,t\}$. If 
\begin{equation}\label{eq:2}
 \sum_{i=1}^k\lambda_i^+<\ga_H,
\end{equation}
then the quadratic form $Q_{\lambda_1,\dots,\lambda_k,a_1,\dots,a_k}$
is positive definite for all $a_1,\dots,a_k\in \R^N$. Conversely,~if
\[
 \sum_{i=1}^k\lambda_i^+>\ga_H
\]
then there exists a configuration of poles $(a_1,\dots,a_k)$ such that $Q_{\lambda_1,\dots,\lambda_k,a_1,\dots,a_k}$ is not positive definite.
\end{theorem}

Finally, it is natural to ask whether $\mu(V)$, defined as an infimum
in \eqref{eq:def_mu}, is attained or not. In the case of a single
pole,  it is known that  the infimum is not achieved, see e.g. \cite{Frank2008a};
however, when dealing with multiple singularities, the outcome can be
different. Indeed, for $V$ in the class $\Theta$, 
we have that $\mu(V) \leq
1-\frac{1}{\ga_H}\max_{i=1,\dots,k,\infty}\la_i$, see Lemma
\ref{Lem-1}, and the infimum is attained in the case of strict
inequality, as 	established in  the following proposition.

\begin{proposition}\label{Prop-1}
	If $V \in\Theta$ is such that
	\begin{equation}\label{Eq}
	\mu(V)<1-\frac{1}{\ga_H}\max\{0,\la_1,\dots, \la_k,\la_\infty \},
	\end{equation}
	then $\mu(V)$ is attained. 
\end{proposition}

The paper is organized as follows. In Section \ref{sec:preliminaries}
we recall some known results about spaces involving fractional
derivatives  and weighted spaces in $\R^{N+1}_+$
 and we prove some estimates needed in the rest of the
article. In Section  \ref{sec:proof-theor-refpr} we prove Theorem
\ref{Prop-2}.  
 In Section \ref{sec:criterion} we prove the positivity
criterion, i.e. Lemma \ref{Criterion}, while in Section
\ref{sec:shattering} we look for upper and lower bounds of the
quantity $\mu(V)$. In Section \ref{sec:perturbation} we investigate
the persistence of the positivity of $\mu(V)$, when the potential $V$
is subject to a perturbation far from the origin or close to a pole. Section
\ref{sec:localization} is devoted to the proof of Theorem
\ref{thm:separation}, that is the primary tool used in the proof of
Theorem \ref{theorem}, pursued in Section
\ref{sec:thm}. Finally, in Section \ref{sec:prop} we prove
Proposition \ref{Prop-1}.

\medskip\noindent{\bf Notation.} 
We list below some notation  used throughout the paper:
\begin{itemize}
 \item[-] $B'(x,r):=\{y\in\R^N\colon \abs{x-y}<r\}$, $B'_r:=B'(0,r)$ for the balls in $\R^N$;
 \item[-] $\R^{N+1}_+:=\{(t,x)\in\R^{N+1}\colon t>0,~x\in\R^N\}$;
 \item[-] $B_r^+:=\{z\in \R^{N+1}_+\colon \abs{z}<r\}$ for the half-balls in $\R^{N+1}_+$;
 \item[-] $\S^N:=\{z\in \R^{N+1}\colon \abs{z}=1\}$ is the unit $N$-dimensional sphere;
 \item[-] $\S^N_+:=\S^N\cap \R^{N+1}_+$ and $\S^{N-1}:=\partial \S^N_+$;
 \item[-] $S_r^+:=\{r\theta\colon \theta \in \S^N_+\}$ denotes a positive half-sphere with arbitrary radius $r>0$;
 \item[-] $\dS$ and $\ds$ denote the volume element in $N$ and $N-1$ dimensional spheres, respectively;
 \item[-] for $t\in\R$, $t^+:=\max\{0,t\}$ and $t^-:=\max\{0,-t\}$.
 
\end{itemize}

\section{Preliminaries}\label{sec:preliminaries}

In this section we clarify some details about the spaces involved in
our exposition and  their relation with the fractional Laplace operator, we recall basic known facts and we prove some introductory results.

\subsection{Preliminaries on Fractional Sobolev Spaces  and  weighted spaces in \texorpdfstring{$\R^{N+1}_+$}{ }}
Let us consider the homogenous Sobolev space $\Ds$ defined in Section \ref{sec:intr}. Thanks to the Hardy-Littlewood-Sobolev inequality
\begin{equation}\label{eq:sobolev}
  S\norm{u}_{L^{2^*_s}(\R^N)}^2\leq \norm{u}_{\Ds}^2,
\end{equation}
that holds for all functions $u\in \Ds$, we have that $\Ds$ is continuously embedded in $L^{2^*_s}(\R^N)$, where $2^*_s:=\frac{2N}{N-2s}$ is the critical Sobolev exponent. 
Combining \eqref{eq:min_ext_ineq} and \eqref{eq:caff_silv} with \eqref{eq:hardy} and \eqref{eq:sobolev}, we obtain, respectively
\begin{equation}\label{eq:hardy_ext}
 \kappa_s \ga_H \int_{\R^N}\frac{\abs{\Tr U}^2}{\abs{x}^{2s}}\dx\leq \int_{\R^{N+1}_+}t^{1-2s}\abs{\nabla U}^2\dxdt\quad \text{for all }U\in \Dext
\end{equation}
and
\begin{equation}\label{eq:sobolev_ext}
 \kappa_s S\norm{\Tr U}_{L^{2^*_s}(\R^N)}^2\leq \int_{\R^{N+1}_+}t^{1-2s}\abs{\nabla U}^2\dxdt\quad \text{for all }U\in \Dext.
\end{equation}
Moreover, just as a consequence of \eqref{eq:min_ext_ineq} and \eqref{eq:caff_silv}, we have that
\begin{equation}\label{eq:dir_princ}
  \kappa_s\norm{\Tr U}_{\Ds}^2\leq
  \int_{\R^{N+1}_+}t^{1-2s}\abs{\nabla U}^2\dxdt \quad
  \text{for all }U\in \Dext.
\end{equation}

Now we state a result providing a compact trace embedding, which will be useful in the following.
\begin{lemma}\label{lemma:compact_emb}
	Let $p\in L^{N/2s}(\R^N)$.  If $(U_n)_n\subseteq \Dext$ and $U\in\Dext$ are such that
	$U_n\deb U$ weakly in $\Dext$ as $n\to\infty$, then
        $\int_{\R^N}p|\Tr U_n|^2\dx\to \int_{\R^N}p|\Tr U|^2\dx$
as $n\to\infty$.
 In
        particular, if $p>0$ a.e. in $\R^N$, 
 the trace operator
	\[
	\Dext \hookrightarrow L^2(\R^N;p\dx)
	\]
	is compact, where $L^2(\R^N;p\dx):=\left\{u\in L^1_{\textup{loc}}(\R^N)\colon \int_{\R^N}p\abs{u}^2\dx<\infty \right\}$.
\end{lemma}
\begin{proof}
	Let $(U_n)_n\subseteq \Dext$ and $U\in\Dext$ be such 
	$U_n\deb U$ weakly in $\Dext$ as $n\to\infty$.		
	Hence, in view of continuity of the trace operator
	\eqref{eq:3} and classical compactness results for fractional
	Sobolev spaces (see e.g. \cite[Theorem
	7.1]{valdinoci-dinezza-palatucci}), we have that 
	$\Tr U_n\to \Tr U$ in $L^2_{\rm loc}(\R^N)$ and a.e. in $\R^N$.
	Furthermore, by continuity of the trace operator
	\eqref{eq:3} and \eqref{eq:sobolev_ext}, we have that, for every
	$\omega\subset\R^N$ measurable, 
	\[
	\int_{\omega}|p||\Tr(U_n-U)|^2\dx\leq C\|p\|_{L^{N/(2s)}(\omega)}
	\|U_n-U\|^2_{\Dext}
	\]
	for some positive constant $C>0$ independent of $\omega$ and $n$.
	Therefore, by Vitali's Convergence Theorem we can conclude
        that 
$\lim_{n\to\infty}\int_{\R^N}|p||\Tr(U_n-U)|^2\dx=0$, from which the
conclusion follows.
\end{proof}
We finally introduce a class of weighted Lebesgue and Sobolev spaces,
on bounded open 
Lipschitz sets $\omega\subseteq \R^{N+1}_+$ in the upper half-space. Namely, we define
\[
	L^2(\omega;t^{1-2s}):=\left\{U\colon \omega\to \R~\text{measurable: such that }\int_{\omega}t^{1-2s}\abs{U}^2\dxdt <\infty  \right\}
\]
and the weighted Sobolev space
\[
	H^1(\omega;t^{1-2s}):=\left\{U\in L^2(\omega;t^{1-2s})\colon \nabla U\in L^2(\omega;t^{1-2s}) \right\}.
\]
From the fact that the weight $t^{1-2s}$ belongs to the second
Muckenhoupt class (see, for instance, \cite{Fabes1982,Fabes1990}) and
thanks to well known weighted inequalities, one can prove that the
embedding $H^1(\omega;t^{1-2s})\hookrightarrow L^2(\omega;t^{1-2s})$
is compact, see for details \cite[Proposition 7.1]{FF-preprint} and \cite{OO}. In addition, in the particular case of $\omega=B_r^+$ one can prove that the trace operators
\begin{gather*}
	H^1(B_r^+;t^{1-2s})\hookrightarrow L^2(B_r'),\quad
	H^1(B_r^+;t^{1-2s})\hookrightarrow L^2(S_r^+;t^{1-2s}),
\end{gather*}
 are well defined and compact, where
 \[
 L^2(S_r^+;t^{1-2s}):=\left\{\psi \colon S_r^+\to
   \R~\text{measurable}:
\int_{S_r^+}t^{1-2s}\abs{\psi }^2\dS <\infty  \right\}.
 \]

\subsection{The Angular Eigenvalue Problem}
Let us consider, for any $\lambda\in\R$, the problem
\begin{equation}\label{eq:angular_eigen}
 \begin{bvp}
  -\divergence_{\S^N} (\theta_1^{1-2s}\nabla_{\S^N} \psi)&=\mu\theta_1^{1-2s}\psi,  && \text{in }\S^{N}_+ ,\\
  -\lim_{\theta_1\to 0^+}\theta_1^{1-2s}\nabla_{\S^N}\psi\cdot \bm{e}_1&=\kappa_s\lambda\psi, &&\text{on }\S^{N-1},
 \end{bvp}
\end{equation}
where $\bm{e}_1=(1,0,\dots,0)\in \R^{N+1}_+$ and $\nabla_{\S^N}$
denotes the gradient on the unit $N$-dimensional sphere $\S^N$. In
order to  give a variational formulation of
\eqref{eq:angular_eigen}   we introduce the following Sobolev space
\[
	H^1(\S^N_+;\theta_1^{1-2s}):=\left\{\psi\in
        L^2(\S^N_+;\theta_1^{1-2s})\colon 
\int_{\S_N^+}\theta_1^{1-2s}|\nabla_{\S^N}\psi|^2\,dS<+\infty \right\}.
\]
  We say that $\psi\in H^1(\S^N_+;\theta_1^{1-2s})$ and $\mu\in\R$ weakly solve \eqref{eq:angular_eigen} if
\begin{equation*}
 \int_{\S^N_+}\theta_1^{1-2s}\nabla_{\S^N}\psi(\theta)\cdot\nabla_{\S^N}\varphi(\theta)\dS=\mu\int_{\S^N_+}\theta_1^{1-2s}\psi(\theta)\varphi(\theta)\dS+\kappa_s\lambda\int_{\S^{N-1}}\psi(0,\theta')\varphi(0,\theta')\ds
\end{equation*}
for all $\varphi\in H^1(\S^N_+;\theta_1^{1-2s})$. 
 By standard spectral arguments,   if
$\lambda<\gamma_H$, there exists a diverging sequence of real eigenvalues of problem \eqref{eq:angular_eigen}
\begin{equation*}
 \mu_1(\lambda)\leq \mu_2(\lambda)\leq \cdots\leq \mu_n(\lambda)\leq \cdots
\end{equation*}
Moreover, each eigenvalue has finite multiplicity (which is counted in
the enumeration above) and
$\mu_1(\lambda)>-\left( \frac{N-2s}{2} \right)^2$ (see \cite[Lemma 2.2]{Fall2014}). For every
$n\geq 1$ we choose an eigenfunction
$\psi_n\in H^1(\S^N_+;\theta_1^{1-2s})\setminus\{0\}$, corresponding
to $\mu_n(\lambda)$, such that
$\int_{\S^N_+}\theta_1^{1-2s}\abs{\psi_n}^2\dS=1$. In addition, we
choose the family $\{\psi_n\}_n$ in such a way that it is an
orthonormal basis of $L^2(\S^N_+;\theta_1^{1-2s})$. We refer to
\cite{Fall2014} for further details.

In \cite{Fall2014} the following implicit characterization of
$\mu_1(\lambda)$ is given. For any $\alpha\in\left( 0,\frac{N-2s}{2} \right)$ we define
\begin{equation}\label{eq:def_lambda_alpha}
 \Lambda(\alpha):=2^{2s}\frac{\Gamma\left(\frac{N+2s+2\alpha}{4}\right)\Gamma\left(\frac{N+2s-2\alpha}{4}\right)}{\Gamma\left(\frac{N-2s+2\alpha}{4}\right)\Gamma\left(\frac{N-2s-2\alpha}{4}\right)}.
\end{equation}
It is known (see e.g. \cite{Frank2008a} and  \cite[Proposition
2.3]{Fall2014}) that the map $\alpha\mapsto  \Lambda(\alpha)$ is
continuous and monotone decreasing. Moreover 
\begin{equation}\label{eq:8}
 \lim_{\alpha\to 0^+}\Lambda(\alpha)=\gamma_H,\qquad \lim_{\alpha\to \frac{N-2s}{2}}\Lambda(\alpha)=0,
\end{equation}
see Figure 1. Furthermore, in \cite[Proposition 2.3]{Fall2014} it is proved that, 
for every $\alpha\in\left( 0,\frac{N-2s}{2} \right)$,
\begin{equation}\label{eq:4}
 \mu_1(\Lambda(\alpha))=\alpha^2-\left( \frac{N-2s}{2} \right)^{\!2}.
\end{equation}
\begin{figure}\label{fig:1}\tiny
     \begin{tikzpicture} 
   \draw [->] (0,0) -- (0,2) node [above] {$\Lambda(\alpha)$};
    \draw [->] (-1,0.5) -- (1.5,0.5) node [below] {$\alpha$};
        \draw [thick] (0,1.3) to [out=350,in=120] (1.17,0.5);
       \node [left] at (0,1.3) {$\ga_H$};
    \node [below] at (0.77,0.5) {$\frac{N-2s}{2}$};
    \draw [fill] (1.17,0.5) circle [radius=0.04];
    \draw [fill] (0,1.3) circle [radius=0.04];
    \end{tikzpicture}
\caption{The function $\Lambda$}    
\end{figure}
In particular, for every $\lambda\in(0,\gamma_H)$ there exists one
and only one $\alpha\in \big(0,\frac{N-2s}{2}\big)$ such that
$\Lambda(\alpha)=\lambda$ and hence $\mu_1(\lambda)=\alpha^2-\big (\frac{N-2s}{2}\big)^2<0$.

We recall the following result from \cite {Fall2018}.

\begin{lemma}[{\cite[Lemma 4.1]{Fall2018}}]\label{lemma:func_fall}
 For every $\alpha\in\left( 0,\frac{N-2s}{2} \right)$ there exists
  $\Upsilon_\alpha:\overline{\R^{N+1}_+}\setminus\{0\}\to\R$
such that  $\Upsilon_\alpha$ is  locally H\"older
                continuous in $\overline{\R^{N+1}_+}\setminus\{0\}$,
  $\Upsilon_\alpha>0$ in $\overline{\R^{N+1}_+}\setminus\{0\}$,  and
  \begin{equation}\label{eq:phi_alpha}
  \begin{bvp}
   -\divergence(t^{1-2s}\nabla \Upsilon_\alpha)&=0, && \text{in }\R^{N+1}_+, \\
   \Upsilon_\alpha(0,x)&=\abs{x}^{-\frac{N-2s}{2}+\alpha}, && \text{on }\R^N, \\
   -\lim_{t\to 0}t^{1-2s}\frac{\partial \Upsilon_\alpha}{\partial t}&=\kappa_s\Lambda(\alpha)\abs{x}^{-2s}\Upsilon_\alpha, && \text{on }\R^N,
  \end{bvp}
 \end{equation}
 in a weak sense. Moreover, $\Upsilon_\alpha\in H^1(B_R^+;t^{1-2s})$ for every $R>0$.
\end{lemma}

The first eigenvalue  $\mu_1(\lambda)$ satisfies the properties described
in the following lemma. 
\begin{lemma}\label{lemma:mu_psi_1}
 Let $\lambda<\gamma_H$. Then the first eigenvalue of problem
 \eqref{eq:angular_eigen} can be characterized as 
 \begin{equation*}
   \mu_1(\lambda)=\inf_{\substack{ \psi\in H^1(\S^N_+;\theta_1^{1-2s}) \\ \psi\nequiv 0 } }\frac{\displaystyle \int_{\S^N_+}\theta_1^{1-2s}\abs{\nabla_{\S^N}\psi}^2\dS-\kappa_s\lambda\int_{\S^{N-1}}\abs{\psi}^2\ds }{\displaystyle \int_{\S^N_+} \theta_1^{1-2s}\abs{\psi}^2\dS}
 \end{equation*}
 and the above infimum is attained by
 $\psi_1\in H^1(\S^N_+;\theta_1^{1-2s})$, which weakly solves
 \eqref{eq:angular_eigen} for $\mu=\mu_1(\lambda)$. Moreover
 \begin{enumerate}
  \item $\mu_1(\lambda)$ is simple, i.e. if $\psi$ attains $\mu_1(\lambda)$ then $\psi=\delta\psi_1$ for some $\delta\in\R$;
  \item either $\psi_1>0$ or $\psi_1<0$ in $\S^N_+$;
  \item if $\lambda>0$ and $\psi_1>0$, then the trace of $\psi_1$ on $\S^{N-1}$ is positive and constant;
  \item if $\lambda=0$ then $\psi_1$ is constant in $\S^N_+$.
 \end{enumerate}
\end{lemma}
\begin{proof}
 The proof of the fact that $\mu_1(\lambda)$ is reached is classical,
 as well as the proofs of points (1) and (2), see for instance
 \cite[Section 8.3.3]{Salsa2016}.
 
 In order to prove (3), let us first  observe that, if
 $\lambda\in(0,\gamma_H)$, there exists one and only one
 $\alpha\in \big(0,\frac{N-2s}{2}\big)$ such that
 $\Lambda(\alpha)=\lambda$. For this $\alpha$ let $\Upsilon_\alpha>0$
 be the solution of \eqref{eq:phi_alpha}. Thanks to \cite[Theorem
 4.1]{Fall2014}, it is possible to describe the behaviour of
 $\Upsilon_\alpha$ near the origin: in particular, since $\Upsilon_\alpha>0$,
 we have that there exists $C>0$ such that
 \begin{equation*}
  \tau^{-a_{\Lambda(\alpha)}}\Upsilon_\alpha(0,\tau\theta')\to C\psi_1(0,\theta')\qquad\text{in }C^{1,\beta}(\S^{N-1})\quad\text{as }\tau\to 0^+,
 \end{equation*}
 where
 \begin{equation*}
  a_{\Lambda(\alpha)}=-\frac{N-2s}{2}+\sqrt{\left(\frac{N-2s}{2}\right)^2+\mu_1(\Lambda(\alpha))}
 \end{equation*}
 Thanks to \eqref{eq:4} we have that
 $a_{\Lambda(\alpha)}=-\frac{N-2s}{2}+\alpha $; then
 $\tau^{-a_{\Lambda(\alpha)}}\Upsilon_\alpha(0,\tau\theta')\equiv 1$ and so
 $\psi_1(0,\theta')$ is positive and constant in $\S^{N-1}$.
 
 Finally, if  $\lambda=0$ then $\mu_1(0)=0$ is clearly attained by every constant function.
\end{proof}

\noindent
 We note that, in view of well known regularity results (see
Proposition \ref{prop:jlx} in the Appendix), $\psi_1\in C^{0,\beta}(
\overline{\S^N_+})$ for some $\beta\in (0,1)$.
Hereafter, we choose the first eigenfunction $\psi_1$ of
problem \eqref{eq:angular_eigen} to be strictly positive in $\S^N_+$. 
With this choice of $\psi_1$, we also have that, in view of the Hopf type
 principle proved in \cite[Proposition 4.11]{Cabre2014} (see Proposition \ref{prop:cabre_sire}),
 \begin{equation}\label{eq:5}
   \min_{\overline{\S^N_+}}\psi_1>0.
 \end{equation}

\subsection{Asymptotic Estimates of Solutions}

In this section, we describe the asymptotic behaviour of solutions to
equations of the type $-\divergence (t^{1-2s}\nabla \Phi)=0$, with
singular potentials appearing in the   Neumann-type boundary
conditions,  either  on positive half-balls $B_r^+$ or on their complement in $\R^{N+1}_+$.

\begin{lemma}\label{lemma:regul_1}
 Let $R_0>0$, $\lambda<\gamma_H$ and let $\Phi\in 
 H^1(B_{R_0}^+;t^{1-2s})$, $\Phi\geq 0$ a.e. in $ B_{R_0}^+$, $\Phi\nequiv 0$, be a weak solution of the following problem
 \begin{equation*}
  \begin{bvp}
    -\divergence (t^{1-2s}\nabla \Phi)&=0, && \text{in }B_{R_0}^+, \\
  -\lim_{t\to 0}t^{1-2s}\frac{\partial\Phi}{\partial t}&=\kappa_s(\lambda\abs{x}^{-2s}+q)\Phi, && \text{on }B_{R_0}',
  \end{bvp}
 \end{equation*}
 where $q\in C^1(B_{R_0}'\setminus\{0\})$ is such that
 \begin{equation*}
  \abs{q(x)}+\abs{x\cdot\nabla q(x)}=O(\abs{x}^{-2s+\eps}) \qquad\text{as }\abs{x}\to 0,
 \end{equation*}
 for some $\eps>0$. Then there exist $C_1>0$ and $R\leq R_0$
 such that
 \begin{equation}\label{eq:regul_1}
  \frac{1}{C_1}\abs{z}^{ a_\lambda}\leq \Phi(z)\leq C_1\abs{z}^{a_\lambda} \qquad\text{for all }z\in B_R^+ ,
 \end{equation}
 where 
\begin{equation}\label{eq:6}
  a_\lambda=-\frac{N-2s}{2}+\sqrt{\left( \frac{N-2s}{2} \right)^{\!2}+\mu_1(\lambda)}.
\end{equation}
Furthermore, if $0\leq \lambda<\gamma_H$, then there exists $C_2>0$ 
such that 
 \begin{equation}\label{eq:regul_2}
  \lim_{\abs{x}\to 0}\abs{x}^{  -a_\lambda}\Phi(0,x)=C_2.
 \end{equation}
\end{lemma}
\begin{proof}
 Since $\Phi\geq 0$ a.e., $\Phi\nequiv 0$, from \cite[Theorem 4.1]{Fall2014} we know that there exists $C>0$ such that
 \begin{equation}\label{eq:7}
  \tau^{-a_\lambda}\Phi(\tau\theta)\to C\psi_1(\theta)\qquad \text{in }C^{0,\beta}(\overline{\S^N_+})\quad\text{as }\tau\to 0.
 \end{equation}
 Estimate \eqref{eq:regul_1} follows from the above convergence
and \eqref{eq:5}. 
Convergence \eqref{eq:regul_2} follows from \eqref{eq:7} and
statements (3--4)  of Lemma
\ref{lemma:mu_psi_1}.
\end{proof}

\begin{lemma}\label{lemma:regul_2}
  Let $R_0>0$, $\lambda<\gamma_H$ and let $\Phi\in\Dext$, $\Phi\geq 0$ a.e. in $\R^{N+1}_+$, $\Phi\nequiv 0$, be a weak solution of the following problem
 \begin{equation*}
  \begin{bvp}
    -\divergence (t^{1-2s}\nabla \Phi)&=0, && \text{in }\R^{N+1}_+\setminus B_{R_0}^+, \\
  -\lim_{t\to 0}t^{1-2s}\frac{\partial\Phi}{\partial t}&=\kappa_s(\lambda\abs{x}^{-2s}+q)\Phi, && \text{on }\R^N\setminus B_{R_0}',
  \end{bvp}
 \end{equation*}
 where $q\in C^1(\R^N\setminus B_{R_0}')$ is such that
 \begin{equation*}
  \abs{q(x)}+\abs{x\cdot\nabla q(x)} =O( \abs{x}^{-2s-\eps}) \qquad\text{as }\abs{x}\to +\infty,
 \end{equation*}
 for some $\eps>0$. Then there exist $C_3>0$ and $R\geq R_0$  such that
 \begin{equation*}
  \frac{1}{C_3}\abs{z}^{  -(N-2s)-a_\lambda}\leq \Phi(z)\leq C_3\abs{z}^{  -(N-2s)-a_\lambda} \qquad\text{for all }z\in \R^{N+1}_+\setminus B_R^+ .
 \end{equation*}
  Furthermore, if $0\leq \lambda<\gamma_H$, then there exists $C_4>0$ 
such that  
 \begin{equation*}
  \lim_{\abs{x}\to \infty}\abs{x}^{  N-2s+a_\lambda}\Phi(0,x)=C_4.
 \end{equation*}
\end{lemma}
\begin{proof}
 The proof follows by considering the equation solved by the Kelvin
 transform of $\Phi$ 
\begin{equation}\label{eq:4kel}
\tilde{\Phi}(z):=\abs{z}^{-(N-2s)}\Phi\bigg(\frac{z}{\abs{z}^2}\bigg)
\end{equation}
(see \cite[Proposition 2.6]{Fall2012}) and applying
 Lemma \ref{lemma:regul_1}.
\end{proof}

\begin{lemma}\label{lemma:regul_3}
 Let $R_0>0$  and let $\Phi\in
 H^1(B_{R_0}^+;t^{1-2s})$, $\Phi\geq 0$ a.e. in $\R^{N+1}_+$, $\Phi\nequiv 0$, be a weak solution of the following problem
 \begin{equation*}
  \begin{bvp}
    -\divergence (t^{1-2s}\nabla \Phi)&=0, && \text{in }B_{R_0}^+, \\
  -\lim_{t\to 0}t^{1-2s}\frac{\partial\Phi}{\partial t}&=\kappa_s q\Phi, && \text{on }B_{R_0}',
  \end{bvp}
 \end{equation*}
 where $q\in C^1(B_{R_0}'\setminus\{0\})$ is such that
 \begin{equation*}
  \abs{q(x)}+\abs{x\cdot\nabla q(x)}  =O(\abs{x}^{-2s+\eps})\qquad\text{as }\abs{x}\to 0,
 \end{equation*}
 for some $\eps>0$. Then there exists $C_5>0$ such that
 \begin{equation*}
  \lim_{\abs{z}\to 0}\Phi(z)=\lim_{\abs{x}\to 0}\Phi(0,x)=C_5.
 \end{equation*}
\end{lemma}
\begin{proof}
  The thesis  is a direct consequence of the regularity result 
of \cite[Proposition 2.4]{Jin2014} (see Proposition \ref{prop:jlx} in
the Appendix) combined with the Hopf type
 principle in \cite[Proposition 4.11]{Cabre2014} (see Proposition
 \ref{prop:cabre_sire}). It can be also derived as a particular case
 of \cite[Theorem 4.1]{Fall2014} with $\lambda=0$, taking into account
 that, for $\lambda=0$, $\psi_1$ is a positive constant on
 $\S^N_+$, as observed in Lemma \ref{lemma:mu_psi_1}.
\end{proof}

\begin{lemma}\label{lemma:regul_4}
  Let $R_0>0$  and let $\Phi\in\Dext$, $\Phi\geq 0$ a.e. in $\R^{N+1}_+$, $\Phi\nequiv 0$, be a weak solution of the following problem
 \begin{equation*}
  \begin{bvp}
    -\divergence (t^{1-2s}\nabla \Phi)&=0, && \text{in }\R^{N+1}_+\setminus B_{R_0}^+, \\
  -\lim_{t\to 0}t^{1-2s}\frac{\partial\Phi}{\partial t}&=\kappa_s q\Phi, && \text{on }\R^N\setminus B_{R_0}',
  \end{bvp}
 \end{equation*}
 where $q\in C^1(\R^N\setminus B_{R_0}')$ is such that
 \begin{equation*}
  \abs{q(x)}+\abs{x\cdot\nabla q(x)} =O(\abs{x}^{-2s-\eps})\qquad\text{as }\abs{x}\to +\infty,
 \end{equation*}
 for some $\eps>0$. Then there exists $C_6>0$ such that
 \begin{equation*}
  \lim_{\abs{z}\to \infty}\abs{z}^{N-2s}\Phi(z)=\lim_{\abs{x}\to \infty}\abs{x}^{N-2s}\Phi(0,x)=C_6 
 \end{equation*}
\end{lemma}
\begin{proof}
   The proof follows by considering    the equation solved by the
 Kelvin transform of $\Phi$ given in \eqref{eq:4kel} and applying Lemma \ref{lemma:regul_3}.
\end{proof}

\section{Proof of Theorem \ref{Prop-2}}\label{sec:proof-theor-refpr}

\begin{proof}[Proof of Theorem \ref{Prop-2}]
First, assume $\sum_{i=1}^k\lambda_i^+<\ga_H$. By Hardy inequality \eqref{eq:hardy} we deduce that
\[
 Q_{\lambda_1,\dots,\lambda_k,a_1,\dots,a_k} (u)\geq \left( 1-\frac{\sum_{i=1}^k\lambda_i^+}{\ga_H} \right)\norm{u}_{\Ds}^2\quad\text{for all }u\in\Ds,
\]
thus implying that $Q_{\lambda_1,\dots,\lambda_k,a_1,\dots,a_k}$ is positive definite.

Now we assume that $\sum_{i=1}^k\lambda_i^+>\ga_H$. By optimality of the constant $\ga_H$ in Hardy inequality, it follows that there exists $\varphi\in C_c^\infty(\R^N)$ such that
\begin{equation}
 \norm{\varphi}_{\Ds}^2-\sum_{i=1}^k\lambda_i^+\int_{\R^N}\frac{\abs{\varphi}^2}{\abs{x}^{2s}}\dx<0.
\end{equation}
Let $\varphi_\rho(x):=\rho^{-\frac{N-2s}{2}}\varphi(x/\rho)$. Then, taking into account Lemma \ref{lemma:conv_hardy}, we have that
\begin{align*}
 \norm{\varphi_\rho}_{\Ds}^2-\sum_{i=1}^k\lambda_i^+\int_{\R^N}\frac{\abs{\varphi_\rho}^2}{\abs{x-a_i}^{2s}}\dx  &=\norm{\varphi}_{\Ds}^2-\sum_{i=1}^k\lambda_i^+\int_{\R^N}\frac{\abs{\varphi}^2}{\abs{x-a_i/\rho}^{2s}}\dx\\
 &\to \norm{\varphi}_{\Ds}^2-\sum_{i=1}^k\lambda_i^+\int_{\R^N}\frac{\abs{\varphi}^2}{\abs{x}^{2s}}\dx<0,
\end{align*}
as $\rho\to+\infty$. Therefore, there exists $\tilde{\rho}$ such that
\begin{equation}\label{eq:prop2_1}
 \norm{\psi}_{\Ds}^2-\sum_{i=1}^k\lambda_i^+\int_{\R^N}\frac{\abs{\psi}^2}{\abs{x-a_i}^{2s}}\dx<0,
\end{equation}
where $\psi:=\varphi_{\tilde{\rho}}$. Let $R>0$ be such that $\supp \psi\subset B'_R$. Then
\begin{equation}\label{eq:prop2_2}
 \int_{\R^N}\frac{\abs{\psi}^2}{\abs{x-a}^{2s}}\leq \frac{1}{(\abs{a}-R)^{2s}}\int_{B'_R}\abs{\psi}^2\dx
\end{equation}
for $\abs{a}$ sufficiently large. Hence, from \eqref{eq:prop2_1} and
\eqref{eq:prop2_2} it follows that 
\begin{multline*}
  Q_{\lambda_1,\dots,\lambda_k,a_1,\dots,a_k}(\psi)=\norm{\psi}_{\Ds}^2-\sum_{i=1}^k
\lambda_i^+\int_{\R^N}\frac{\abs{\psi}^2}{\abs{x-a_i}^{2s}}\dx+
\sum_{i=1}^k\lambda_i^-\int_{\R^N}\frac{\abs{\psi}^2}{\abs{x-a_i}^{2s}}\dx \\
  \leq
  \norm{\psi}_{\Ds}^2-\sum_{i=1}^k\lambda_i^+
\int_{\R^N}\frac{\abs{\psi}^2}{\abs{x-a_i}^{2s}}\dx+\sum_{i=1}^k\lambda_i^-
\frac{1}{(\abs{a_i}-R)^{2s}}\int_{B'_R}\abs{\psi}^2\dx<0
\end{multline*}
if the poles $a_i$'s, corresponding to negative $\lambda_i$'s, are sufficiently far from the origin. The proof is thereby complete.
\end{proof}

\begin{remark}\label{rmk:2_poles}
 We observe that, in the case of two poles (i.e. $k=2$), Theorem
 \ref{Prop-2} implies the sufficiency of condition \eqref{lambda-i}
 for the existence of a configuration of poles that makes the
 quadratic form $Q_{\lambda_1,\dots,\lambda_k,a_1,\dots,a_k}$ positive
 definite. Indeed, if $k=2$ condition \eqref{lambda-i} directly
 implies  \eqref{eq:2}.
\end{remark}

\section{A positivity criterion in the class \texorpdfstring{$\Theta$}{Theta}}\label{sec:criterion}

In this section, we provide the proof of Lemma \ref{Criterion}, that is a criterion for establishing positivity of Schr\"odinger operators with potentials in $\Theta$, in relation with existence of positive supersolutions, in the spirit of Allegretto-Piepenbrink theory.

We first prove the equivalent formulation of the infimum in
\eqref{eq:def_inf} stated in Lemma \ref{lemma:equiv_inf}.

\begin{proof}[Proof of Lemma \ref{lemma:equiv_inf}]
 Let's fix $\tilde{u}\in \Ds\setminus \{0\}$ and let's call $\tilde{U}\in \Dext$ its extension. Since
 \begin{equation*}
  \kappa_s\norm{\tilde{u}}_{\Ds}^2=\int_{\R^{N+1}_+}t^{1-2s}|\nabla \tilde{U}|^2 \dxdt,
 \end{equation*}
 where $\kappa_s$ is defined in \eqref{eq:kappa_s}, then
 \begin{equation*}
  \frac{ \norm{\tilde{u}}_{\Ds}^2-\int_{\R^N}V
    \tilde{u}^2\dx}{\norm{\tilde{u}}_{\Ds}^2 }=
\frac{\int_{\R^{N+1}_+}t^{1-2s}|\na \tilde{U}|^2\dxdt-\kappa_s\int_{\Rn}V|\Tr \tilde{U}|^2\dx}{\int_{\R^{N+1}_+}t^{1-2s}|\na \tilde{U}|^2\dxdt}.
 \end{equation*}
 Therefore
 \begin{equation*}
  \frac{ \norm{\tilde{u}}_{\Ds}^2-\int_{\R^N}V \abs{\tilde{u}}^2\dx}{\displaystyle\norm{\tilde{u}}_{\Ds}^2 }
  \geq \inf_{\substack{U \in \Dext \\ U\nequiv 0}}\frac{\int_{\R^{N+1}_+}t^{1-2s}|\na U|^2\dxdt-\kappa_s\int_{\Rn}V|\Tr U|^2\dx}{\int_{\R^{N+1}_+}t^{1-2s}|\na U|^2\dxdt}
 \end{equation*}
 and then we can pass to the inf also on the left-hand quotient.

On the other hand, from \eqref{eq:dir_princ}, we have that, for any $U \in \Dext\setminus\{0\}$
 \begin{equation*}
  \frac{\int_{\R^{N+1}_+}t^{1-2s}|\na U|^2\dxdt-\kappa_s\int_{\Rn}V|\Tr U|^2\dx}{\int_{\R^{N+1}_+}t^{1-2s}|\na U|^2\dxdt}\geq 1-\frac{\int_{\R^N}V \abs{u}^2\dx}{\norm{u}_{\Ds}^2}
\end{equation*}
where $u=\Tr U$. Taking the infimum to both sides concludes the proof. 
\end{proof}

Now we are able to provide the proof of the positivity criterion.

\begin{proof}[Proof of Lemma \ref{Criterion}]
	Let us first prove (I).  Let
	$U\in C_c^\infty (\overline{\R^{N+1}_+}\setminus
	\{(0,a_1),\dots,(0,a_k)\})\setminus \{ 0 \}$,
	$U\nequiv 0$ on $\R^N$. Note that $U^2/\Phi\in \Dext$ and $U^2/\Phi
	\geq 0$, hence we can choose $U^2/\Phi$ in \eqref{eq:1} as a test function.  
	Easy computations
	yield
	\begin{equation*}
		\int_{\R^{N+1}_+} t^{1-2s}\abs{\nabla U}^2\dxdt-\kappa_s\int_{\R^N} V\abs{\Tr U}^2\dx \geq \eps \kappa_s \int_{\R^{N}} \tilde{V}\abs{\Tr U}^2\dx
	\end{equation*}
	which, taking into account the hypothesis on $\tilde{V}$, implies that
	\begin{equation*}
		\kappa_s \int_{\R^N} V\abs{\Tr U}^2\dx \leq \frac{1}{1+\eps}\int_{\R^{N+1}_+} t^{1-2s}\abs{\nabla U}^2\dxdt.
	\end{equation*}
	Hence
	\begin{equation}\label{eq:inf_eps}
	\frac{\int_{\R^{N+1}_+} t^{1-2s}\abs{\nabla U}^2\dxdt-\kappa_s\int_{\R^N} V\abs{\Tr U}^2\dx }{ \int_{\R^{N+1}_+} t^{1-2s}\abs{\nabla U}^2\dxdt}\geq \frac{\eps}{1+\eps}.
	\end{equation}
	Therefore (I) follows by density of $C_c^\infty (\overline{\R^{N+1}_+}\setminus
	\{(0,a_1),\dots,(0,a_k)\})$ in $\Dext$  (see Lemma \ref{l:density}).
	
	Now we prove (II). First of all we notice that, thanks to H\"older's inequality, \eqref{eq:hardy_ext} and \eqref{eq:sobolev_ext}
	\begin{multline*}
		\kappa_s\int_{\R^N}\tilde{V}\abs{\Tr U}^2\dx \leq\kappa_s\int_{\R^N}\abs{V}\abs{\Tr U}^2\dx \\ \leq  \left[ \frac{1}{\gamma_H}\left(\sum_{i=1}^k\abs{\lambda_i}+\abs{\lambda_\infty}\right) +S^{-1}\norm{W}_{L^{N/2s}(\R^N)}\right]\int_{\R^{N+1}_+}t^{1-2s}\abs{\nabla U}^2 \dxdt 
	\end{multline*}
	for all $U\in \Dext$. If
	\begin{equation}\label{eq:12}
	0 <\eps < \frac{\mu(V)}{2} \left[ \frac{1}{\gamma_H}\left(\sum_{i=1}^k\abs{\lambda_i}+\abs{\lambda_\infty}\right) +S^{-1}\norm{W}_{L^{N/2s}(\R^N)}\right]^{-1},
      \end{equation}
	then
	\begin{equation}\label{eq:criterion_2}
	\int_{\R^{N+1}_+} t^{1-2s}\abs{\nabla U}^2\dxdt -\kappa_s\int_{\R^N}( V+\eps \tilde{V})\abs{\Tr U}^2\dx 
	\geq \frac{\mu(V)}{2}\int_{\R^{N+1}_+} t^{1-2s}\abs{\nabla U}^2\dxdt
	\end{equation}
	for all $U\in \Dext$.
	Hence, for any fixed $p\in L^{N/2s}(\R^N)\cap L^\infty (\R^N)$, H\"older continuous and positive, the infimum
	\begin{equation*}
		m_p=\inf_{\substack{U\in \Dext\\\Tr U\not\equiv 0 }}
		\frac{\displaystyle \int_{\R^{N+1}_+} t^{1-2s}\abs{\nabla
				U}^2\dxdt-\kappa_s\int_{\R^N}( V+\eps \tilde{V} )\abs{\Tr
				U}^2\dx }{\displaystyle \int_{\R^{N}} p \abs{\Tr
				U}^2\dx }
	\end{equation*}
	is nonnegative. Also $m_p$ is achieved by some function $\Phi
	\in \Dext \setminus\{0\}$, that (by evenness) can be
	chosen to be nonnegative: indeed, thanks to Hardy inequality \eqref{eq:hardy_ext} and \eqref{eq:criterion_2} it's easy to prove that the map
	\begin{equation*}
		\Dext \ni U\mapsto \int_{\R^{N+1}_+} t^{1-2s}\abs{\nabla U}^2\dxdt
		-\kappa_s\int_{\R^N}
		( V+\eps \tilde{V})\abs{\Tr U}^2\dx
	\end{equation*}
	is weakly lower semicontinuous (since its square root is an equivalent
		norm in $\Dext$), while Lemma \ref{lemma:compact_emb} yields the
	compactness of the trace map from $\Dext$ into $L^2(\R^N,p\dx)$.
	Moreover $\Phi$ satisfies in a weak sense
	\begin{equation}\label{pbm:criterion_1}
	\begin{bvp}
	-\divergence (t^{1-2s}\nabla \Phi)&=0, &&\text{in }\R^{N+1}_+ ,\\
	-\lim_{t\to 0^+} t^{1-2s}\frac{\partial \Phi}{\partial t}&= \kappa_s(V+\eps \tilde{V})\Tr \Phi + m_p p\Tr \Phi, &&\text{in }\R^N,
	\end{bvp}
	\end{equation}
	i.e. 
	\begin{equation*}
		\int_{\R^{N+1}_+}t^{1-2s}\na \Phi\cdot \na W\dxdt = \kappa_s \int_{\R^{N}}
		(V+\eps \tilde{V})\Tr\Phi\Tr W\dx+m_p\int_{\R^{N}}
		p\Tr\Phi\Tr W\dx
	\end{equation*}
	for all $W \in \Dext$.
	From \cite[Proposition
	2.6]{Jin2014} (see also Proposition \ref{prop:jlx} in the Appendix) we
	have that $\Phi$ is  locally H\"older
                continuous in $\overline{\R^{N+1}_+} \setminus
                \{(0,a_1), \dots, (0,a_k)\}$; in particular $\Phi\in
 C^{0}\left(\overline{\R^{N+1}_+} \setminus
		\{(0,a_1),\dots, (0,a_k)\}\right)$. 
	Moreover, the classical Strong Maximum
	Principle implies that $\Phi>0$ in $\R^{N+1}_+$; then, in the
        case when $V,\tilde{ V}$ are locally H\"older
                continuous in $\R^N\setminus\{a_1,\dots,a_k\}$, the Hopf type
	principle proved in \cite[Proposition 4.11]{Cabre2014} (which is
	recalled in the Proposition \ref{prop:cabre_sire} of the Appendix) ensures that
	$\Phi(0,x)>0$ for all $x\in \R^N \setminus \{a_1, \dots,
        a_k,\}$; we observe that assumption \eqref{eq:9} of
        Proposition \ref{prop:cabre_sire} is satisfied thanks to
        \cite[Lemma 4.5]{Cabre2014}, see Lemma \ref{l:cabre_sire}.
\end{proof}

\section{Upper and lower bounds for $\mu(V)$}\label{sec:shattering}

In this section we prove bounds from above and from below (in Lemma \ref{Lem-1} and \ref{Lem-2}, respectively) for the quantity $\mu(V)$.

\begin{lemma}\label{Lem-1}
For any $V(x)=\sum_{i=1}^{k}\frac{\la_i \chi_{B'(a_i,r_i)}(x)}{|x-a_i|^{2s}}+\frac{\la_\infty \chi_{\Rn \setminus B'_R}(x)}{|x|^{2s}}+W(x) \in \Theta,$ there holds:
\begin{itemize}
\item [(i)] $\mu(V) \leq 1;$
\item [(ii)] if $\max_{i=1,\dots,k,\infty}\la_i>0,$ then $\mu(V) \leq 1-\frac{1}{\ga_H}\max_{i=1,\dots,k,\infty}\la_i.$	
\end{itemize}		
\end{lemma}	

\begin{proof}
Let us fix
  $u \in C_c^{\infty}(\Rn)$, $u\not\equiv 0$ and
  $P\in\R^N\setminus\{a_1,\dots,a_k\}$. For every $\rho>0$ we define
  $u_{\rho}(x):=\rho^{-\frac{(N-2s)}{2}}u(\frac{x-P}{\rho})$ and we
   notice that, by scaling properties,
	\begin{equation}\label{eq:shatter_1}
          \norm{u_\rho}_{\Ds}=\norm{u}_{\Ds}\quad\text{and}
\quad\norm{u_\rho}_{L^{2^*_s}(\R^N)}=\norm{u}_{L^{2^*_s}(\R^N)}.
	\end{equation}
	Moreover, since $\supp(u_\rho)=P+\rho \supp(u)$, we have that
        $a_1,\dots,a_k\not\in \supp(u_\rho)$ for $\rho>$ sufficiently small, hence
	\begin{equation}\label{eq:shatter_2}
		V\in L^{N/2s}(\supp(u_\rho)).
	\end{equation}
	Therefore, from the definition of $\mu(V)$, thanks also to
        \eqref{eq:shatter_1}, \eqref{eq:shatter_2}, H\"older
        inequality, and \eqref{eq:sobolev}, we deduce that
	\begin{align*}
          \mu(V)&\leq
          1-\frac{\int_{\R^N}V\abs{u_\rho}^2\dx}{\norm{u_\rho}_{\Ds}^2}
\leq 1+\frac{\int_{\supp(u_\rho)}\abs{V}\abs{u_\rho}^2\dx}{\norm{u}_{\Ds}^2}\\
&          \leq
          1+\frac{\norm{V}_{L^{N/2s}(\supp(u_\rho))}\norm{u}^2_{L^{2^*_s}(\R^N)}}{\norm{u}_{\Ds}^2}
\leq 1+S^{-1}\norm{V}_{L^{N/2s}(\supp(u_\rho))}=
          1+o(1),
        \end{align*}
	as $\rho\to 0^+$. By density we may conclude the first part of the proof.

Now let us assume $\max_{i=1, \dots,k,\infty}\lambda_i>0$ and let us
first consider the case $\max_{i=1, \dots,k,\infty}\lambda_i=\lambda_j$ for a
certain $j=1,\dots,k$. From optimality of the best constant in Hardy
inequality \eqref{eq:hardy} and from the density of
$C_c^\infty(\R^N \setminus \{a_1, \dots, a_k,0\})$ in
$\mathcal{D}^{s,2}(\R^N)$ (see Lemma \ref{l:density}), we have that, for any $\eps>0$, there exists
$\varphi\in C_c^\infty(\R^N \setminus \{a_1, \dots, a_k,0\})$ such that
\begin{equation}\label{eq:shatter_3}
	\norm{\varphi}_{\Ds}^2<(\gamma_H+\eps)\int_{\R^N}\frac{\abs{\varphi}^2}{\abs{x}^{2s}}\dx.
\end{equation}
Now, for any $\rho>0$ we define $\varphi_\rho(x):=\rho^{-\frac{(N-2s)}{2}}\varphi(\frac{x-a_j}{\rho})$. From the definition of $\mu(V)$ and from \eqref{eq:shatter_1} we deduce that
\begin{equation}\label{eq:shatter_4}
	  \mu(V) \leq 1-\frac{\int_{\R^{N}}V\abs{\varphi_\rho}^2\dx}{\norm{\varphi}_{\Ds}^2}.
\end{equation}
On the other hand, we can split the numerator as
\begin{align*}
	\int_{\R^{N}}V\abs{\varphi_\rho}^2\dx&=\la_j \int_{B'(a_j,r_j)}|x-a_j|^{-2s}\abs{\varphi_\rho}^2\dx+\sum_{i\neq j}\la_i \int_{B'(a_i,r_i)}|x-a_i|^{-2s}\abs{\varphi_\rho}^2\dx\\
	&+\la_\infty \int_{\Rn \setminus B'_{R}}|x|^{-2s}\abs{\varphi_\rho}^2\dx+\int_{\R^N}W\abs{\varphi_\rho}^2\dx.
\end{align*}
From H\"older inequality and \eqref{eq:shatter_1} we have that
\begin{equation}\label{eq:shatter_5}
	 \left| \int_{\Rn}W\varphi_\rho^2\dx\right| \leq \norm{W}_{L^{N/2s}(\supp(\varphi_\rho))}\norm{\varphi}_{{L^{2^*_s}}}^2 \to 0, \quad\text{as }\rho\to 0^+,
\end{equation}
while, just by a change of variable
\begin{equation}\label{eq:shatter_6}
	\int_{B'(a_j,r_j)}|x-a_j|^{-2s}\abs{\varphi_\rho}^2\dx=\int_{B'_{r_j/\rho}}|x|^{-2s}\abs{\varphi}^2\dx\to \int_{\R^{N}} |x|^{-2s}\abs{\varphi}^2\dx,
\end{equation}
as $\rho\to 0^+$. Moreover $\supp(\varphi_\rho)=a_j+\rho\supp(\varphi)$, and therefore, thanks to \eqref{eq:shatter_5} and \eqref{eq:shatter_6}, we have that, as $\rho\to 0^+$,
\begin{equation}\label{eq:shatter_7}
\int_{\R^{N}}V\abs{\varphi_\rho}^2\dx=\la_j \int_{\R^N}|x|^{-2s}\abs{\varphi}^2\dx+o(1).
\end{equation}
Hence, combining \eqref{eq:shatter_4} with \eqref{eq:shatter_7} and \eqref{eq:shatter_3}, we obtain that
\[
	\mu(V)\leq 1-\lambda_j(\gamma_H+\eps)^{-1},
\]
for all $\eps>0$, which implies that $\mu(V)\leq
1-\lambda_j/\gamma_H$. Finally, let us assume $\max_{i=1,
  \dots,k,\infty}\lambda_i=\lambda_\infty$. Letting
$\varphi_\rho(x):=\rho^{-\frac{(N-2s)}{2}}\varphi(x/\rho )$, we
observe that $\varphi_\rho\to 0$ uniformly, as $\rho\to +\infty$. So,
arguing as before, one can similarly obtain that
$\mu(V) \leq 1-\frac{\la_\infty}{\ga_H}$.
The proof is thereby complete.
\end{proof}

The following result provides the positivity in the case of
potentials with subcritical masses supported in sufficiently small
neighbourhoods of the poles. 
In the following we fix two cut-off functions
$\zeta,\tilde\zeta:\R^N\to \R$ such that $\zeta,\tilde\zeta\in
C^\infty(\R^N)$, $0\leq \zeta(x)\leq 1$, $0\leq \tilde\zeta(x)\leq 1$,
and 
\begin{align*}
&\zeta(x)=1\quad\text{for $|x|\leq\frac12$},
\quad \zeta(x)=0\quad\text{for $|x|\geq1$},\\
&\tilde \zeta(x)=0\quad\text{for $|x|\leq1$},
\quad \tilde\zeta(x)=1\quad\text{for $|x|\geq2$}.
\end{align*}
\begin{lemma}\label{Lem-2}
	Let $\{a_1,a_2,\dots,a_k\}\subset B_R'$, $a_i\neq a_j$
        for $i\neq j$,  and $\lambda_1,\lambda_2,\dots,\lambda_k,\lambda_\infty\in\R$
        be such that $m:=\max_{i=1, \dots,k,\infty}\la_i<\gamma_H$.  
 For any $0<h<1-\frac{m}{\ga_H}$, there exists $\de=\de(h)>0$ such that 
	\begin{equation*}
	\mu\left(\sum_{i=1}^{k}\frac{\la_i
            \zeta(\frac{x-a_i}\delta)}{|x-a_i|^{2s}}+\frac{\la_\infty
\tilde\zeta(\frac xR)}{|x|^{2s}} \right) 
	 \geq 
	\begin{cases}
	1-\frac{m}{\ga_H}-h,& \text{if } m>0\\
	1,&\text{if } m \leq 0.
	\end{cases}
	\end{equation*}
\end{lemma}
\begin{proof}
 Let us assume that $m>0$, otherwise the statement is trivial. First, let us fix $0<\eps < \frac{\gamma_H}{ m}-1,$ so that
\begin{equation*}
  \tilde{\la}_i:=\la_i +\eps \la_i^+ < \gamma_H \quad\text{for all }i=1,\dots,k,\infty. 
\end{equation*}
In order to prove the statement, it is sufficient to find
$\delta=\delta(\eps)>0$ and 
$\Phi \in \Dext$ such that $\Phi\in
C^{0}(\overline{\R^{N+1}_+}\setminus
\{(0,a_1/\de),\dots,(0,a_k/\de)\})$, $\Phi> 0$ in
$\overline{\R^{N+1}_+}\setminus \{(0,a_1/\de),\dots,(0,a_k/\de)\}$, and
\begin{equation}\label{Eq-1}
\int_{\R^{N+1}_+}t^{1-2s}\na \Phi\cdot\na U\dxdt-\sum_{i=1}^k \kappa_s\int_{\R^{N}}V_i\Tr \Phi \Tr U\dx -\kappa_s\int_{\R^{N}}V_\infty \Tr \Phi \Tr U\dx \geq 0
\end{equation} 
for all $U \in \Dext$, $U \geq 0$ a.e., where
\[
 V_i(x)=\frac{\tilde{\la}_i
   \zeta\big(x-\frac{a_i}\delta\big)}{\abs{x-a_i/\de}^{2s}},\qquad
 V_\infty(x)=\frac{\tilde{\la}_\infty \tilde\zeta\big(\frac{\delta}R x\big)}{\abs{x}^{2s}}.
\]
Indeed, thanks to scaling properties in \eqref{Eq-1} and to Lemma
\ref{Criterion}, \eqref{Eq-1} implies that 
\[
 \mu \left( \sum_{i=1}^k \frac{\la_i 
 \zeta(\frac{x-a_i}\delta)}{\abs{x-a_i}^{2s}}+\frac{\la_\infty 
\tilde\zeta(\frac xR)}{\abs{x}^{2s}} \right) \geq\frac{\eps}{1+\eps},
\]
so that,  letting 
$h:=1-\frac{m}{\gamma_H}-\frac{\eps}{\eps+1}$, 
we obtain the result. Hence, we seek for some $\Phi$ positive and
continuous in $\overline{\R^{N+1}_+}\setminus
\{(0,a_1/\de),\dots,(0,a_k/\de)\}$ satisfying \eqref{Eq-1}. Let us set, for some $0<\tau<1$,
\[
  p_i(x):=p\left(x-\frac{a_i}{\de}\right)\quad \text{for }
  i=1,\dots,k,\quad 
p_\infty(x)=\left(\frac{\delta}R \right)^{2s}p\left(\frac{\delta x}R\right),
\]
where $p(x)=\frac{1}{\abs{x}^{2s-\tau}(1+\abs{x}^2)^{\tau}}$.
We observe that $p_i, p_\infty \in L^{N/{2s}}(\Rn)$. Therefore, thanks to Lemma \ref{lemma:compact_emb}, the weighted eigenvalue 
\[
 \mu_i=\inf_{\substack{\Phi \in \Dext \\  \Tr\Phi\nequiv 0 }}\frac{\int_{\R^{N+1}_+} t^{1-2s} \abs{\na \Phi}^2\dxdt-\kappa_s\int_{\Rn}V_i\abs{\Tr \Phi }^2\dx}{\int_{\R^{N}}p_i\abs{\Tr \Phi }^2\dx}
\]
is positive and attained by some nontrivial, nonnegative function $\Phi_i \in \Dext$ that weakly solves
\[
 \begin{bvp}
  -\divergence(t^{1-2s}\nabla\Phi_i)&=0, &&\text{in }\R^{N+1}_+ ,\\
  -\lim_{t\to 0}t^{1-2s}\frac{\partial \Phi_i}{\partial t}&=\left(
    \kappa_s V_i+\mu_i p_i  \right)\Tr\Phi_i ,&&\text{on }\R^N.
 \end{bvp}
\]
From the classical Strong Maximum Principle we deduce that $\Phi_i>0$
in $\R^{N+1}_+$, while Proposition \ref{prop:jlx} yields that 
$\Phi_i$ is locally H\"older
                continuous in  $\overline{\R^{N+1}_+}\setminus
\{(0,a_i/\de)\}$.
Moreover, from the Hopf type lemma in Proposition
\ref{prop:cabre_sire} (whose assumption \eqref{eq:9} is satisfied
thanks to  Lemma \ref{l:cabre_sire} outside $\{a_i/\de\}$) we deduce that $\Tr \Phi_i >0$ in
$\R^N\setminus \{a_i/\de\}$.  Similarly
\[
 \mu_\infty=\inf_{\substack{\Phi \in \Dext \\ \Tr\Phi\nequiv 0 }}\frac{\int_{\R^{N+1}_+}t^{1-2s}\abs{\na \Phi}^2\dxdt-\kappa_s\int_{\R^{N}}V_\infty\abs{\Tr \Phi }^2\dx}{\int_{\R^{N}}p_\infty\abs{\Tr \Phi }^2\dx}
\]
is positive and reached by some nontrivial, nonnegative function
$\Phi_\infty \in \Dext$ such that $\Phi_\infty$ 
is locally H\"older
                continuous in  $\overline{\R^{N+1}_+}$ and 
$\Phi_\infty>0$ in $\overline{\R^{N+1}_+}$.  Moreover, $\Phi_{\infty}$ weakly solves
\[
 \begin{bvp}
  -\divergence(t^{1-2s}\nabla\Phi_\infty)&=0, &&\text{in }\R^{N+1}_+ ,\\
  -\lim_{t\to 0}t^{1-2s}\frac{\partial \Phi_\infty}{\partial
    t}&=( \kappa_s V_\infty+\mu_\infty p_\infty  )\Tr\Phi_\infty ,&&\text{on }\R^N.
 \end{bvp}
\]
Lemmas \ref{lemma:regul_1}--\ref{lemma:regul_4} (and continuity of the
$\Phi_i$'s outside the poles)  imply that 
there exists $C_0>0$ (independent of $\de$) such that
\begin{align}
  \frac{1}{C_0}\abs{x-\frac{a_i}{\de}}^{a_{\tilde \lambda_i}} 
\leq &\Tr \Phi_i \leq C_0 \abs{x-\frac{a_i}{\de}}^{a_{\tilde \lambda_i}}, && \text{in }B'(a_i/\de,1), \label{eq:mu_estim_1}\\
  \frac{1}{C_0}\abs{x-\frac{a_i}{\de}}^{-(N-2s)} \leq &\Tr \Phi_i \leq C_0 \abs{x-\frac{a_i}{\de}}^{-(N-2s)}, &&\text{in }\Rn \setminus B'(a_i/\de,1), \label{eq:mu_estim_2}\\
  \frac{1}{C_0}\abs{\frac{\de x}{R}}^{-(N-2s)-a_{\tilde\lambda_\infty}} \leq &\Tr \Phi_\infty \leq C_0 
\abs{\frac{\de x}{R}}^{-(N-2s)-a_{\tilde\lambda_\infty}}, && \text{in }\Rn \setminus B'_{R/\de}, \label{eq:mu_estim_3}\\
  \frac{1}{C_0} \leq &\Tr \Phi_\infty \leq C_0,&& \text{in }B'_{R/\de}. \label{eq:mu_estim_4}
\end{align}
Let $\Phi:=\sum_{i=1}^k \Phi_i + \eta \Phi_\infty$, with $0<\eta<\inf \{\frac{\mu_i}{4C_0^2 \tilde{\la}_i}\colon i=1,\dots,k,~ \tilde{\la}_i>0 \}$.
Therefore,
\begin{align}
 & \int_{\R^{N+1}_+}t^{1-2s}\na \Phi\cdot \na U\dxdt-\sum_{i=1}^k\int_{\Rn}V_i\Tr \Phi \Tr U \dx-\int_{\R^{N}}V_\infty\Tr \Phi \Tr U \dx \nonumber \\
	=&\int_{\R^N}\bigg[ \sum_{i=1}^k\bigg(\mu_ip_i-V_\infty-\sum_{ j\neq i} V_j\bigg)\Tr \Phi_i +\eta\bigg( \mu_\infty p_\infty-\sum_{i=1}^k V_i \bigg)\Tr \Phi_\infty \bigg]\Tr U \dx \nonumber\\
	=:&\int_{\R^N}g(x)\Tr U(x) \dx \label{eq:mu_estim_5} 
\end{align}
for all $U\in\Dext$. Hereafter, let us assume $U\geq 0$ a.e. in $\R^{N+1}_+$. We will split the integral into three parts and prove that each of these is nonnegative. First, let us consider $x\in B'_{R/\de}\setminus\left( \cup_{i=1}^k B'(a_i/\de,1) \right)$: here $V_i=V_\infty=0$ and we have that
\[
 g(x)= \sum_{i=1}^k \mu_i p_i \Tr \Phi_i +\eta \mu_\infty p_\infty \Tr \Phi_\infty  \geq 0.
\]
Now let us take $n\in \{1,\dots,k\}$ and $x\in B'(a_n/\de,1)$, where $V_\infty=V_i=0$ for $i\neq n$. Then
\begin{gather*}
 g(x)=\sum_{i=1}^k\mu_i p_i  \Tr \Phi_i -V_n\sum_{ i\neq n}\Tr \Phi_i +\eta (\mu_\infty p_\infty -V_n)\Tr \Phi_\infty  \\
 \geq  \mu_n p_n \Tr \Phi_n - V_n\bigg(\sum_{\substack{i=1 \\i\neq n}}^k\Tr \Phi_i +\eta\Tr \Phi_\infty \bigg).
\end{gather*}
If $\tilde{\lambda}_n\leq 0$ this is clearly nonnegative; so let us assume $\tilde{\lambda}_n>0$. Thanks to \eqref{eq:mu_estim_1}, \eqref{eq:mu_estim_2} and \eqref{eq:mu_estim_4} we can estimate this quantity from below by
\begin{equation}\label{eq:mu_estim_6}
 \abs{x-\frac{a_n}{\de}}^{-2s}\bigg[ 
\frac{\mu_n}{2^\tau
  C_0}\abs{x-\frac{a_n}{\de}}^{\tau+a_{\tilde\lambda_n}}-
\tilde{\lambda}_n C_0\bigg( \sum_{i\neq n}\abs{x-\frac{a_i}{\de}}^{-(N-2s)}+\eta \bigg) \bigg].
\end{equation}
We observe that $\abs{x-a_n/\de}^{\tau+a_{\tilde\lambda_n}}\geq 1$,
since
$\tilde\lambda_n>0$ implies that $a_{\tilde\lambda_n}<0$ and we
can 
choose $\tau<-a_{\tilde\lambda_n}$.
Moreover it's not hard to prove that, for $i\neq n$,
\[
 \abs{x-\frac{a_i}{\de}}^{-(N-2s)}\leq \bigg( \frac{2}{\abs{a_n-a_i}} \bigg)^{N-2s}\delta^{N-2s}<\frac{\eta}{k-1},
\]
for $\de>0$ sufficiently small. Thanks to this and to the choice of
$\eta$ we have that 
the expression in \eqref{eq:mu_estim_6} and then 
$g(x)$ is nonnegative in $B'(a_n/\de,1)$ . Finally, if
$x\in\R^N\setminus B'_{R/\de}$, then
 the function $g$ in  \eqref{eq:mu_estim_5} becomes
\begin{equation}\label{eq:mu_estim_7}
  \sum_{i=1}^k( \mu_i p_i -V_\infty )\Tr \Phi_i +\eta\mu_\infty p_\infty \Tr \Phi_\infty  .
\end{equation}
Again, if $\tilde{\lambda}_\infty\leq 0$ this quantity is
nonnegative. If $\tilde{\lambda}_\infty>0$, thanks to
\eqref{eq:mu_estim_2} and \eqref{eq:mu_estim_3}, we have that the
function in \eqref{eq:mu_estim_7} is greater than or equal to
\begin{equation}\label{eq:mu_estim_8}
 \abs{x}^{-2s}\bigg[ -C_0\tilde{\lambda}_\infty
 \sum_{i=1}^k\abs{x-\frac{a_i}{\de}}^{-(N-2s)}+\frac{\eta\mu_\infty}{2^\tau
   C_0}\abs{\frac{\de x}{R}}^{-(N-2s)-a_{\tilde\lambda_\infty}+\tau} \bigg].
\end{equation}
Now, one can easily see that
\[
 \abs{x-\f{a_i}{\de}} \geq \bigg(1-\frac{a}{R}\bigg)\abs{x}\qquad \text{for all } x \in \Rn \setminus B'_{R/\de}, \quad \text{where } a=\max_{j=1,\dots,k}\abs{a_j},
\]
so that we can estimate \eqref{eq:mu_estim_8} from below obtaining
that, for all $x\in\R^N\setminus B'_{R/\de}$,
\begin{gather*}
g(x)\geq  \abs{x}^{-N}\bigg[ -C_0\tilde{\lambda}_\infty k \bigg( 1-\frac{a}{R} \bigg)^{-(N-2s)}+ \frac{\eta\mu_\infty}{2^\tau C_0}\abs{\frac{\de}{R}}^{-(N-2s)-a_{\tilde\lambda_\infty}+\tau} \abs{x}^{-a_{\tilde\lambda_\infty}+\tau}\bigg] \\
  \geq \abs{x}^{-N}\bigg[ -C_0\tilde{\lambda}_\infty k \bigg( 1-\frac{a}{R} \bigg)^{-(N-2s)}+ \frac{\eta\mu_\infty}{2^\tau C_0}\abs{\frac{\de}{R}}^{-(N-2s)} \bigg]\geq 0
\end{gather*}
for $\de>0$ sufficiently small, since
$a_{\tilde\lambda_\infty}<0$ if $\tilde\lambda_\infty>0$.  The proof is thereby complete.
\end{proof}

\section{Perturbation at infinity and at poles}\label{sec:perturbation}
In this section,  we investigate the persistence of the positivity when
the mass is increased at infinity (Theorem \ref{Theorem-a}) and at
poles (Theorem \ref{Theorem-a-pole}). 

In order to make use of Lemmas \ref{lemma:regul_1}--\ref{lemma:regul_4}, we may need to restrict the class $\Theta$ to some more regular potentials and to have a control on their growth at infinity.

For any $\de>0$, we define 
 \begin{equation}\label{eq:def_P_infty}
 \begin{aligned}
  \mathcal{P}_\infty^\de:=\bigg\{ f\colon \R^N\to\R \colon &f\in C^1(\R^N\setminus B_{R_\infty}')\text{ for some $R_\infty>0$}\\[-10pt]
  &\text{and}~\abs{f(x)}+\abs{x\cdot \nabla f(x)}=O(\abs{x}^{-2s-\de})~\text{as }\abs{x}\to +\infty\bigg\}.
  \end{aligned}
 \end{equation}
Moreover, in order to prove some intermediary, technical lemmas based
on the positivity criterion Lemma \ref{Criterion},  the need for even more regular potentials occasionally arises. So, let us introduce the class
\begin{equation}\label{eq:def_Theta_s}
	\Theta^*:=\left\{V\in\Theta\colon V\in C^1(\R^N\setminus\{a_1,\dots,a_k\}) \right\}.
\end{equation}
Then, we will recover the full generality of the class $\Theta$, thanks to an approximation procedure, which is based on the following lemma.
\begin{lemma}\label{lemma:approx_potentials}
	Let $V_1,V_2\in\Theta$ be such that $V_1-V_2\in L^{N/2s}(\R^N)$. Then
	\[
		\abs{\mu(V_2)-\mu(V_1)}\leq S^{-1}\norm{V_2-V_1}_{L^{N/2s}(\R^N)},
	\]
	where $S>0$ is the best constant in the Sobolev embedding \eqref{eq:sobolev}.
\end{lemma}
\begin{proof}
	From the definition of $\mu(V_2)$, H\"older inequality and
        \eqref{eq:sobolev_ext}, 
 for all $U\in \Dext$ we have that
	\begin{multline}
		\int_{\R^{N+1}_+}t^{1-2s}\abs{\nabla U}^2\dxdt-\kappa_s\int_{\R^N}V_1\abs{\Tr U}^2\dx \\
		\geq \left(\mu(V_2)-S^{-1}\norm{V_2-V_1}_{L^{N/2s}(\R^N)} \right)\int_{\R^{N+1}_+}t^{1-2s}\abs{\nabla U}^2\dxdt,
	\end{multline}
	which implies that
	\[
		\mu(V_1)\geq \mu(V_2)-S^{-1}\norm{V_2-V_1}_{L^{N/2s}(\R^N)}.
	\]
	Analogously one can prove that
	$\mu(V_2)\geq \mu(V_1)-S^{-1}\norm{V_2-V_1}_{L^{N/2s}(\R^N)}$,
	thus concluding the proof.
\end{proof}

\begin{lemma}\label{lemma-a}
Let $V\in\mathcal H$,  $a_1,\dots,a_k\in\R^N$,
and $R>0$ be such that 
\[
V\in C^1(\R^N\setminus\{a_1,\dots,a_k\}) \quad\text{and}\quad 
V(x)=\frac{\la_\infty}{|x|^{2s}}+W(x) \quad\text{in }\Rn \setminus B'_{R},
\]
where $\lambda_\infty<\gamma_H$ and $W\in\mathcal{P}_\infty^\de\cap
L^\infty(\R^N)$ for some $\de>0$.
  Assume that
$\mu(V)>0$ and let $\nu_\infty \in \R$ be such that
$\la_\infty+\nu_\infty<\ga_H$. Then there exist $\tilde{R}>R$ and 
$\Phi \in \Dext$ such that $\Phi$ is  locally H\"older
                continuous in $\overline{\R^{N+1}_+} \setminus \{(0,a_1),\dots,(0,a_k)\}$, $\Phi>0$ in
$\overline{\R^{N+1}_+} \setminus \{(0,a_1),\dots,(0,a_k)\}$,
and
\[
	\int_{\R^{N+1}_+}t^{1-2s}\na \Phi\cdot \na U\dxdt-\kappa_s\int_{\Rn}\left[V+ \frac{\nu_\infty}{|x|^{2s}}\chi_{\Rn \setminus B'_{\tilde{R}} }\right]\Tr \Phi  \Tr U \dx\geq 0,
\]
for all $U \in \Dext$ with $U\geq 0$ a.e.
\end{lemma}	

\begin{proof}
By \eqref{eq:8} we can fix $\eps\in\big(0,\frac{N-2s}2\big)$  such that
\begin{equation}\label{eq:cond_eps}
	\Lambda(\eps)-\lambda_\infty>0\quad\text{and}\quad \Lambda(\eps)-\lambda_\infty-\nu_\infty>0.
\end{equation}
 Since $W\in\mathcal{P}_\infty^\de\cap L^\infty(\R^N)$, there
exists  $C_0>0$ such that 
\begin{equation}\label{eq:cond_W}
	W(x) \leq \frac{C_0}{|x|^{2s+\de}} \qquad\text{in } \Rn.
\end{equation}
Let $R_0\geq
\max\Big\{R,\frac{1}{2}\left[\frac{C_0}{\Lambda(\eps)-\lambda_\infty}\right]^{{1}/{\de}}
\Big\}$, so that
\begin{equation}\label{eq:cond_R_0}
	\Lambda(\eps)-\lambda_\infty-C_0(2R_0)^{-\de}\geq 0.
\end{equation}
From Lemma \ref{lemma:func_fall} there exists a positive,  locally
 H\"older continuous function
$\Upsilon_\eps: \overline{\R^{N+1}_+}\setminus\{0\} \to \R$ such that 
$\Upsilon_\eps\in\bigcap_{r>0}H^1(B_r^+;t^{1-2s})$  and
\begin{equation}
  \left\{\begin{aligned}
      -{\divergence}(t^{1-2s}\na \Upsilon_\eps), &=0 &&\text{in } \R^{N+1}_+, \\
      \Upsilon_\eps (0,x)&=|x|^{-\f{N-2s}{2}+\eps}, &&\text{on } \Rn,\\
      -\lim_{t \to 0^+}t^{1-2s}\frac{\pa \Upsilon_\eps}{\pa t},&=\kappa_s\Lambda(\eps)|x|^{-2s}\Tr \Upsilon_\eps &&\text{on } \Rn,\end{aligned}
\right.
\end{equation}
in a weak sense.
 Direct calculations (see e.g. \cite[Proposition
2.6]{Fall2012}) yield  that the Kelvin transform
\[
  \tilde{\Upsilon}_\eps(z)=\abs{z}^{-(N-2s)}\Upsilon_\eps(z/\abs{z}^2)                                                                                    
\]
 of $\Upsilon_\eps$ weakly satisfies
\begin{equation}\label{kelvin-trans}
\left\{\begin{aligned}
-{\divergence}(t^{1-2s}\na \tilde{\Upsilon}_\eps) &=0, && \text{in }\R^{N+1}_+\setminus\{0\}, \\
\tilde{\Upsilon}_\eps (0,x)&=|x|^{\f{2s-N}{2}-\eps}, &&\text{on } \Rn \setminus\{0\},\\
-\lim_{t \to 0^+}t^{1-2s}\frac{\pa \tilde{\Upsilon}_\eps}{\pa t}&=\kappa_s\Lambda(\eps)|x|^{-2s}\Tr \tilde{\Upsilon}_\eps, &&\text{on } \Rn \setminus\{0\},
\end{aligned}
\right.
\end{equation}
 $\tilde{\Upsilon}_\eps>0$ in $\overline{\R^{N+1}_+}\setminus\{0\}$
 and $\tilde{\Upsilon}_\eps$ 
 is   locally H\"older
                continuous in $\overline{\R^{N+1}_+} \setminus
                \{0\}$. Moreover we have that 
\begin{equation}\label{eq:10}
  \int_{\R^{N+1}_+ \setminus B_r^+} t^{1-2s} |\nabla\tilde{\Upsilon}_\eps|^2\dxdt+
  \int_{\R^{N+1}_+\setminus B_r^+} t^{1-2s}
  \frac{|\tilde{\Upsilon}_\eps|^2}{|x|^2+t^2}\dxdt<+\infty\quad\text{for
    all }r>0.
\end{equation}
Let $\eta\in C^\infty(\overline{\R^{N+1}_+})$ be a cut-off function
such that $\eta$ is radial, i.e. $\eta(z)=\eta(\abs{z})$, $\abs{\nabla
  \eta}\leq \frac2{R_0}$ in $\overline{\R^{N+1}_+}$, 
\[
	\eta(z):=\begin{cases}
	0, & \text{in }B^+_{R_0}\cup B_{R_0}' \\
	1, & \text{in }\left(\R^{N+1}_+\setminus B_{2R_0}^+\right)\cup\left(\R^N\setminus B_{2R_0}'  \right) ,
	\end{cases}
\]
and $\eta>0$ in $\overline{\R^{N+1}_+}\setminus \overline{B_{R_0}^+}$. We point out that
\begin{equation*}
	\frac{\partial \eta}{\partial t}(0,x)=0\quad \text{and}\quad
        \frac{1}{t}\abs{\frac{\partial \eta}{\partial
            t}(t,x)}=O(1)\quad\text{as }t\to 0 \text{ (uniformly in $x$)}.
\end{equation*}
We let $\Phi_1:=\eta \tilde{\Upsilon}_\eps$.
By its construction, $\Phi_1$ is continuous on the whole
$\overline{\R^{N+1}_+}$ and $\Phi_1>0$ in $\overline{\R^{N+1}_+}\setminus \overline{B_{R_0}^+}$,
whereas  \eqref{eq:10} implies that $\Phi_1\in \Dext$. 
 Moreover direct computations yield that $\Phi_1$  weakly solves
\begin{equation}\label{phi-1}
	\left\{\begin{aligned}
	-{\divergence}(t^{1-2s}\na \Phi_1) &= t^{1-2s}F_1 ,& &\text{in }\R^{N+1}_+, \\
	-\lim_{t \to 0^+}t^{1-2s}\frac{\pa \Phi_1}{\pa t}&=\kappa_s\Lambda(\eps)|x|^{-2s}\Tr \Phi_1, &&\text{on }\R^N,
	\end{aligned}\right.
\end{equation}
where 
\[
F_1:=(2s-1)\frac{1}{t}\frac{\partial \eta}{\partial
    t}\tilde{\Upsilon}_\eps-2\nabla \tilde{\Upsilon}_\eps\cdot\nabla
  \eta-\tilde{\Upsilon}_\eps \Delta \eta.
\]
We observe that   $F_1\in C^\infty(\R^{N+1}_+)$ and $\supp(F_1)
\subset \overline{B_{2R_0}^+ \setminus B_{R_0}^+}$.
Given 
\[
f_1(x):=\kappa_s\Lambda(\eps)\abs{x}^{-2s}\chi_{B_{2R_0}'\setminus
  B_{R_0}'}\Phi_1(0,x),
\]
 we can choose a smooth, compactly supported function $f_2\colon \R^N\to \R$ such that 
\begin{equation}\label{eq:f_1_f_2}
	f_1+f_2\geq 0\quad\text{in }\R^N,\qquad 	H:=f_2+\left[W+\lambda_\infty \abs{x}^{-2s} \right]\chi_{B_{2R_0}'\setminus B_{R_0}'} \Tr\Phi_1\geq 0 \quad\text{in }\R^N.
      \end{equation}
We also choose another smooth, positive, compactly supported function $F_2\colon \overline{\R^{N+1}_+}\to \R$ such that $F_1+F_2 \geq 0$ in $\overline{\R^{N+1}_+}$. Since $\mu(V)>0$ and $H \in L^{\frac{2N}{N+2s}}(\Rn)$, by Lax-Milgram Lemma there exists $\Phi_2 \in \Dext$ such that 
\begin{equation}\label{phi-2}
\left\{\begin{aligned}
-{\divergence}(t^{1-2s}\na \Phi_2) &=t^{1-2s}F_2 , &&\text{in }\R^{N+1}_+, \\
-\lim_{t \to 0^+}t^{1-2s}\frac{\pa \Phi_2}{\pa t}&=\kappa_s\left[V\Tr\Phi_2+H\right], &&\text{on }\R^N,
\end{aligned}
\right.
\end{equation}
holds in a weak sense. From Proposition \ref{prop:jlx} we
know that  $\Phi_2$ is locally H\"older
                continuous in $\overline{\R^{N+1}_+} \setminus
                \{(0,a_1),\dots,(0,a_k)\}$. 

In order to prove that $\Phi_2$ is strictly positive in
$\overline{\R^{N+1}_+}\setminus \{(0,a_1),\dots,(0,a_k)\}$, we compare
it with the unique weak solution $\Phi_3 \in \Dext$  to the problem 
\begin{equation}\label{eq:11}
\left\{\begin{aligned}
-{\divergence}(t^{1-2s}\na \Phi_3) &=0 , &&\text{in }\R^{N+1}_+, \\
-\lim_{t \to 0^+}t^{1-2s}\frac{\pa \Phi_3}{\pa t}&=\kappa_s\left[V\Tr\Phi_3+H\right], &&\text{on }\R^N,
\end{aligned}
\right.
\end{equation}
whose existence is again ensured by the Lax-Milgram Lemma.
The difference $\widetilde\Phi=\Phi_2-\Phi_3$ belongs to $\Dext$ and weakly solves 
\begin{equation*}
\left\{\begin{aligned}
-{\divergence}(t^{1-2s}\na \widetilde\Phi) &=t^{1-2s}F_2 , &&\text{in }\R^{N+1}_+, \\
-\lim_{t \to 0^+}t^{1-2s}\frac{\pa \widetilde\Phi}{\pa t}&=\kappa_s V\Tr \widetilde\Phi, &&\text{on }\R^N.
\end{aligned}
\right.
\end{equation*}
By directly testing the above equation with
$-\widetilde\Phi^-$, since $\mu(V)>0$ we obtain that
$\widetilde\Phi\geq 0$ in $\R^{N+1}_+$, i.e. $\Phi_2\geq\Phi_3$.
Furthermore, testing the equation for $\Phi_3$ with
$-\Phi_3^-$, we also obtain that $\Phi_3\geq0$  in $\R^{N+1}_+$.
The classical Strong Maximum
Principle, combined with Proposition \ref{prop:cabre_sire} (whose
assumption \eqref{eq:9} for \eqref{eq:11} is satisfied thanks to
the  assumption $V\in C^1(\R^N\setminus\{a_1,\dots,a_k\})$ and 
Lemma \ref{l:cabre_sire}), yields
$\Phi_3>0$ in
$\overline{\R^{N+1}_+}\setminus \{(0,a_1),\dots,(0,a_k)\}$
and hence 
\[
\Phi_2>0\quad\text{in}\quad 
\overline{\R^{N+1}_+}\setminus \{(0,a_1),\dots,(0,a_k)\}.
\]
 Finally,
from Lemma \ref{lemma:regul_2} and from the continuity of $\Phi_2$,
there exists $C_1>0$ such that
\begin{equation}\label{eq:phi_2_estim}
	\frac{1}{C_1}|x|^{-(N-2s)-a_{\lambda_\infty}} \leq
        \Phi_2(0,x) \leq C_1|x|^{-(N-2s)-a_{\lambda_\infty}}
\quad \text{in }\Rn \setminus B'_{2R_0}.
\end{equation}
Now we set $\Phi=\Phi_1+\Phi_2$. We immediately observe that
$\Phi\in\Dext$  is locally H\"older continuous and strictly
positive in
$\overline{\R^{N+1}_+} \setminus \{(0,a_1),\dots,(0,a_k)\}$. 
We claim that, for $\tilde{R}>0$ sufficiently large,
\begin{equation}\label{phi}
\int_{\R^{N+1}_+}t^{1-2s}\na \Phi\cdot \na U\dxdt-\kappa_s\int_{\Rn}\left[V+ \frac{\nu_\infty}{|x|^{2s}}\chi_{\Rn \setminus B'_{\tilde{R}} }\right]\Tr \Phi  \Tr U \dx\geq 0,
\end{equation}
for all $U \in \Dext$ with $U\geq 0$ a.e.

The function $\Phi$ weakly satisfies
\begin{equation}
\left\{\begin{aligned}
-{\divergence}(t^{1-2s}\na \Phi) &=t^{1-2s}(F_1+F_2) , &&\text{in }\R^{N+1}_+, \\
-\lim_{t \to 0^+}t^{1-2s}\frac{\pa \Phi}{\pa t}&=\kappa_s\left[ \Lambda(\eps)\abs{x}^{-2s}\Tr \Phi_1+ V\Tr\Phi_2+H\right], &&\text{on }\R^N.
\end{aligned}
\right.
\end{equation}
Hence, if $U\in\Dext$, $U\geq 0$ a.e.,
\begin{equation*}
	\begin{gathered}
			\int_{\R^{N+1}_+}t^{1-2s}\nabla \Phi\cdot\nabla U\dxdt-\kappa_s\int_{\Rn}\left[V+ \frac{\nu_\infty}{|x|^{2s}}\chi_{\Rn \setminus B'_{\tilde{R}} }\right]\Tr \Phi  \Tr U \dx \\
			\geq \int_{\R^N}\left[\kappa_s\frac{\Lambda(\eps)-\lambda_\infty-\nu_\infty }{ \abs{x}^{2s}}\chi_{\R^N\setminus B'_{\tilde{R}}}\Tr \Phi_1  
			+\kappa_s\frac{\Lambda(\eps) -\lambda_\infty -\abs{x}^{2s}W}{ \abs{x}^{2s}}\chi_{B'_{\tilde{R}}\setminus B'_{2R_0}}\Tr\Phi_1 \right. \\
			\left.-\kappa_s\left(W\Tr\Phi_1+\frac{\nu_{\infty}}{\abs{x}^{2s} }\Tr\Phi_2\right)\chi_{\R^N\setminus B'_{\tilde{R}}} +f_1+f_2\right]\Tr U\dx=:\int_{\R^N}F(x)\Tr U(x)\dx.
	\end{gathered}
\end{equation*}
If $x\in B'_{2R_0}$, then $F(x)=f_1(x)+f_2(x)\geq0$. If $x\in B'_{\tilde{R}}\setminus B'_{2R_0}$, then from \eqref{eq:f_1_f_2}, \eqref{eq:cond_W} and \eqref{eq:cond_R_0} 
\[
	F(x)\geq \kappa_s(\Lambda(\eps) -\lambda_\infty -\abs{x}^{2s}W)\abs{x}^{-2s} \Tr\Phi_1\geq \kappa_s(\Lambda(\eps)-\lambda_\infty-C_0(2R_0)^{-\de}) \abs{x}^{-2s}\Tr\Phi_1\geq 0.
\]
Finally, if $x\in \R^N\setminus B'_{\tilde{R}}$, then from the definition of $\Phi_1$, \eqref{eq:f_1_f_2}, \eqref{eq:phi_2_estim} and \eqref{eq:cond_W} we have that
\[
F(x)\geq \kappa_s(\Lambda(\eps)-\lambda_\infty-\nu_\infty)\abs{x}^{-\frac{N+2s}{2}-\eps}-\kappa_s C_0 \abs{x}^{-\frac{N+2s}{2}-\eps-\de}-\kappa_s C_1 {\nu_\infty\abs{x}^{-N-a_{\lambda_\infty}}}.
\]
Since the function $\lambda\mapsto\mu_1(\lambda)$ is strictly decreasing and
$\lambda_\infty<\Lambda(\eps)$, from \eqref{eq:4} it follows that
$\mu_1(\lambda_\infty)>\eps^2-\big(\frac{N-2s}{2}\big)^2$ which yields
$-N- a_{\lambda_\infty}<-\frac{N+2s}{2}-\eps$. Hence, if $\tilde{R}$
is sufficiently large, $F(x)\geq0$ for all $x\in  \R^N\setminus
B'_{\tilde{R}}$.  This concludes the proof.
\end{proof}	

Combining Lemma \ref{lemma-a} with the  positivity criterion Lemma
\ref{Criterion} and an approximation procedure based on Lemma
\ref{lemma:approx_potentials}, we prove the persistence of the
positivity under perturbations at infinity for potentials in the class $\Theta$.
\begin{theorem}\label{Theorem-a}
Let
\[
 V(x)=\sum_{i=1}^{k}\frac{\la_i \chi_{B'(a_i,r_i)}(x)}{|x-a_i|^{2s}}+\frac{\la_\infty\chi_{\Rn \setminus B'_R}(x)}{|x|^{2s}}+W(x) \in \Theta.
\]
Assume $\mu(V)>0$ and let $\nu_\infty \in \R$ be such that $\la_\infty+\nu_\infty<\ga_H$. Then there exists $\tilde{R}>R$ such that
\[
 \mu\left(V+\frac{\nu_\infty}{|x|^{2s}}\chi_{\R^N \setminus B'_{\tilde{R}}}\right)>0.
\]
\end{theorem}	
\begin{proof}
  Since $V \in \Theta$ and $\mu(V)>0$,  arguing as in
  \eqref{eq:criterion_2} we have that, for $\eps$ chosen sufficiently
  small as in \eqref{eq:12},  $\mu(V+\eps
  V)>\frac{\mu(V)}{2}>0$.
  Moreover we can choose $\eps$ such that
  $\la_\infty +\nu_\infty+\eps(\la_\infty + \nu_\infty)<\ga_H$ 
  and 
$\lambda_i +\eps \lambda_i<\gamma_H$ for all
  $i=1,\dots,k,\infty$. Let $\sigma=\sigma(\epsilon)$ be such that
\begin{equation}\label{eq:pert_infty_1}
	0<\sigma<\min\{S\epsilon,S\mu(V)/2\}.
\end{equation} 
By density of $C^\infty_{\rm c}(\R^N)$ in
$L^{\frac{N}{2s}}(\R^N)$ there exists 
\[
	\hat{V}(x)=\sum_{i=1}^{k}\frac{\la_i \chi_{B'(a_i,r_i)}(x)}{|x-a_i|^{2s}}+\frac{\la_\infty\chi_{\Rn \setminus B'_R}(x)}{|x|^{2s}}+\hat{W}(x) \in \Theta^*
\]
such that $\hat{W}\in\mathcal{P}_\infty^\delta$ for some $\delta>0$ and
\begin{equation}\label{eq:pert_infty_2}
	\lVert\hat{V}-V\rVert_{L^{N/2s}(\R^N)}<\frac{\sigma}{1+\epsilon}.
\end{equation}
Then from Lemma \ref{lemma:approx_potentials}, taking into account \eqref{eq:pert_infty_1} and \eqref{eq:pert_infty_2}, we have that
\[
	\mu(\hat{V}+\eps\hat{V})\geq \mu(V+\eps V)-(1+\eps)S^{-1}\lVert\hat{V}-V\rVert_{L^{N/2s}(\R^N)}>0.
\]
Now, thanks to Lemma \ref{lemma-a}, there exists $\tilde{R}>R$ and a
function
 $\Phi \in \Dext$ such that  $\Phi$ is strictly positive and locally
 H\"older continuous
 in $\overline{\R^{N+1}_+} \setminus \{(0,a_1),\dots,(0,a_k)\}$  and 
\[
\int_{\R^{N+1}_+}t^{1-2s}\na \Phi\cdot \na U\dxdt-\kappa_s\int_{\Rn}\left[\hat{V}+\eps\hat{V}+ \frac{\nu_\infty+\eps \nu_\infty}{|x|^{2s}}\chi_{\Rn \setminus B'_{\tilde{R}} }\right]\Tr \Phi  \Tr U \dx\geq 0,
\]
for all $U \in \Dext$ with $U\geq 0$ a.e. Therefore Lemma \ref{Criterion} yields
\[
	\mu\left( \hat{V}+\frac{\nu_\infty}{\abs{x}^{2s}}\chi_{\R^{N} \setminus B'_{\tilde{ R}}} \right)\geq \frac{\eps}{1+\eps}.
\]
Finally, thanks to Lemma \ref{lemma:approx_potentials},
\eqref{eq:pert_infty_1} 
and \eqref{eq:pert_infty_2}, we have the estimate
\[
	\mu\left( V+\frac{\nu_\infty}{\abs{x}^{2s}}\chi_{\R^{N} \setminus B'_{\tilde{ R}}} \right)\geq \mu\left( \hat{V}+\frac{\nu_\infty}{\abs{x}^{2s}}\chi_{\R^{N} \setminus B'_{\tilde{ R}}} \right)-S^{-1}\lVert\hat{V}-V\rVert_{L^{N/2s}(\R^N)}>0
\]
which yields the conclusion.
\end{proof}	

 Swapping the singularity at a pole for a singularity at infinity
through the Kelvin transform, we obtain the analog of Theorem
\ref{Theorem-a} when perturbing the mass of a pole.

\begin{theorem}\label{Theorem-a-pole}
Let
\[
 V(x)=\sum_{i=1}^{k}\frac{\la_i \chi_{B'(a_i,r_i)}(x)}{|x-a_i|^{2s}}+\frac{\la_\infty\chi_{\Rn \setminus B'_R}(x)}{|x|^{2s}}+W(x) \in \Theta.
\]
Assume $\mu(V)>0$ and let $i_0\in\{1,\dots,k\}$ and $\nu \in \R$ be such that $\lambda_{i_0}+\nu<\ga_H$. Then there exists $\delta\in(0,r_{i_0})$ such that
\[
 \mu\left(V+\frac{\nu}{|x-a_{i_0}|^{2s}}\chi_{B'(a_{i_0},\delta)}\right)>0.
\]
\end{theorem}	
Before proving Theorem \ref{Theorem-a-pole}, it is convenient to make
the following remark.
\begin{remark}\label{rem:translation_kelvin}
  \begin{enumerate}[(i)]
  \item By the invariance by translation of the norm $\|\cdot\|_{\Ds}$,
    we have that, if $V\in\mathcal H$, then, for any $a\in\R^N$, the translated
    potential $V_a:=V(\cdot+a)$ belongs to $\mathcal H$ and $\mu(V_a)=\mu(V)$. 
\item If $V\in \mathcal H$ and
  $V_K(x):=|x|^{-4s}V\big(\frac{x}{|x|^2}\big)$, then $V_K\in \mathcal
  H$ and $\mu(V_K)=\mu(V)$. To prove this statement, we observe that,
  by the change of variables $y=\frac{x}{|x|^2}$,
\[
\int_{\R^N}|V_K(x)|^2u^2(x)\dx=
\int_{\R^N}|V(y)|^2(\mathcal Ku)^2(y)\dy\quad\text{for any }u\in\Ds,
\]
where $(\mathcal Ku)(x):=|x|^{2s-N}u\big(\frac{x}{|x|^2}\big)$ is the
Kelvin transform of $u$. The claim  then follows from the fact that
$\mathcal K$ is an isometry on $\Ds$ (see \cite[Lemma 2.2]{Fall2012}).
  \end{enumerate}
\end{remark}

\begin{proof}[Proof of Theorem \ref{Theorem-a-pole}]
Let $V_1(x):=V(x+a_{i_0})$.  We have that 
\[
 V_1(x)=\frac{\lambda_{i_0} \chi_{B'_{r_{i_0}}}(x)}{|x|^{2s}}+
\sum_{i\neq i_0}\frac{\la_i
  \chi_{B'(a_i-a_{i_0},r_i)}(x)}{|x-(a_i-a_{i_0})|^{2s}}+
\frac{\la_\infty\chi_{\Rn \setminus B'_R}(x)}{|x|^{2s}}+W_1(x) \in \Theta
\]
and, in view
  of Remark \ref{rem:translation_kelvin} (i), $\mu(V_1)=\mu(V)>0$.
  Then  we can choose some $\eps$
  sufficiently small so that $\mu(V_1+\eps
  V_1)>\frac{\mu(V)}{2}>0$ (see 
  \eqref{eq:criterion_2}) and   $\lambda_{i_0} +\nu+\eps(\lambda_{i_0} +\nu)<\ga_H$,
$\lambda_i+\eps \lambda_i<\gamma_H$ for all
  $i=1,\dots,k,\infty$. Let $\sigma=\sigma(\epsilon)\in(0,\min\{S\epsilon,S\mu(V)/2\})$.
By density of $C^\infty_{\rm c}(\R^N\setminus\{a_1,\dots,a_k\})$ in
$L^{\frac{N}{2s}}(\R^N)$ there exists 
\[
	V_2(x)=
\frac{\lambda_{i_0} \chi_{B'_{r_{i_0}}}(x)}{|x|^{2s}}+
\sum_{i\neq i_0}\frac{\la_i
  \chi_{B'(a_i-a_{i_0},r_i)}(x)}{|x-(a_i-a_{i_0})|^{2s}}+
\frac{\la_\infty\chi_{\Rn \setminus B'_R}(x)}{|x|^{2s}}+W_2(x) \in \Theta^*
\]
such that $W_2\in L^{N/2s}(\R^N)\cap
  L^\infty(\R^N)$ vanishes in a neighbourhood of any pole
  and in a  neighbourhood of $\infty$ and
\begin{equation*}
	\|V_2-V_1\|_{L^{N/2s}(\R^N)}<\frac{\sigma}{1+\epsilon}.
\end{equation*}
Let $V_3(x):=|x|^{-4s}V_2\big(\frac{x}{|x|^2}\big)$. Then 
\[
V_3\in
C^1\bigg(\R^N\setminus\left\{0,\tfrac{a_i-a_{i_0}}{|a_i-a_{i_0}|^2}\right\}_{i\neq
  i_0}\bigg)
\]
and there exists $r>0$ such that 
\[
V_3(x)=\frac{\lambda_{i_0}}{|x|^{2s}}
\quad\text{in }\R^N\setminus B_r'.
\]
Moreover, from  Remark \ref{rem:translation_kelvin} (ii) and 
Lemma \ref{lemma:approx_potentials} it follows that $V_3\in\mathcal H$ and 
\begin{align*}
\mu(V_3+\eps V_3)&=\mu(V_2+\eps V_2)
\\
&\geq \mu(V_1+\eps V_1)-S^{-1}(1+\eps)\|V_1-V_2\|_{L^{N/2s}(\R^N)}
>\frac{\mu(V)}2-S^{-1}\sigma>0.
\end{align*}
From Lemma \ref{lemma-a} we deduce that  there exists $\tilde{R}>r$ and a
function
 $\Phi \in \Dext$ such that $\Phi$ is strictly positive and locally
 H\"older continuous
 in $\overline{\R^{N+1}_+} \setminus 
\left\{0,\tfrac{a_i-a_{i_0}}{|a_i-a_{i_0}|^2}\right\}_{i\neq
  i_0}$  and 
\[
\int_{\R^{N+1}_+}t^{1-2s}\na \Phi\cdot \na
U\dxdt-\kappa_s\int_{\Rn}\left[V_3+\eps V_3+ \frac{\nu+\eps \nu}{|x|^{2s}}\chi_{\Rn \setminus B'_{\tilde{R}} }\right]\Tr \Phi  \Tr U \dx\geq 0,
\]
for all $U \in \Dext$ with $U\geq 0$ a.e. Therefore Lemma \ref{Criterion} yields
\[
	\mu\left( V_3+\frac{\nu}{\abs{x}^{2s}}\chi_{\R^{N} \setminus B'_{\tilde{ R}}} \right)\geq \frac{\eps}{1+\eps}.
\]
From  Remark \ref{rem:translation_kelvin} (ii) we have that 
$\mu\left( V_3+\frac{\nu}{\abs{x}^{2s}}\chi_{\R^{N} \setminus
    B'_{\tilde{ R}}} \right)=\mu\left(
  V_2+\frac{\nu}{\abs{x}^{2s}}\chi_{B'_{1/\tilde{ R}}} \right)\geq
\frac{\eps}{1+\eps}$. Hence, letting $\delta=1/\tilde R$,  from
Remark \ref{rem:translation_kelvin} (i) and Lemma
\ref{lemma:approx_potentials} we deduce that 
\begin{align*}
\mu\bigg(V+&\frac{\nu}{|x-a_{i_0}|^{2s}}\chi_{B'(a_{i_0},\delta)}\bigg)=
\mu\left(V_1+\frac{\nu}{|x|^{2s}}\chi_{B'_\delta}\right)\\
&\geq 
\mu\left(V_2+\frac{\nu}{|x|^{2s}}\chi_{B'_\delta}\right)
-S^{-1}\lVert V_1-V_2\rVert_{L^{N/2s}(\R^N)}
\geq \frac\eps{1+\eps}-\frac{S^{-1}\sigma}{1+\eps}>0
\end{align*}
which yields the conclusion.
\end{proof}

\begin{corollary}\label{cor:pert-at-pole}
  Let
\[
 V(x)=\sum_{i=1}^{k}\frac{\la_i \chi_{B'(a_i,r_i)}(x)}{|x-a_i|^{2s}}+\frac{\la_\infty\chi_{\Rn \setminus B'_R}(x)}{|x|^{2s}}+W(x) \in \Theta
\]
be such that $\mu(V)>0$. Then there exists 
\begin{equation}\label{eq:15}
 \widetilde V(x)=\sum_{i=1}^{k}\frac{\tilde \la_i \chi_{B'(a_i,\tilde
     r_i)}(x)}{|x-a_i|^{2s}}+\frac{\la_\infty\chi_{\Rn \setminus
     B'_R}(x)}{|x|^{2s}}+\widetilde W(x) \in \Theta
\end{equation}
such that 
\[
 \widetilde V- V\in C^\infty(\R^N\setminus\{a_1,\dots,a_k\}),\quad
 \widetilde V\geq V,\quad
\mu( \widetilde V)>0,\quad\text{and}\quad 
\tilde \la_i>0\text{ for all }i=1,\dots,k.
\]
\end{corollary}
\begin{proof}
For every $i=1,\dots,k$, let $\nu_i$ be such that
$\nu_i>0$ and $\lambda_i+\nu_i\in(0,\gamma_H)$.  From Theorem
\ref{Theorem-a-pole} we have that, for every $i=1,\dots,k$, there
exists $\delta_i$ such that, letting 
\[
\widehat
V=V+\sum_{i=1}^k\frac{\nu_i}{|x-a_{i}|^{2s}}\chi_{B'(a_{i},\delta_i)},
\]
$\mu(\widehat V)>0$.
Let us consider a  cut-off function
$\zeta:\R^N\to \R$ such that $\zeta\in
C^\infty(\R^N)$, $0\leq \zeta(x)\leq 1$,
$\zeta(x)=1$ for $|x|\leq\frac12$, and $\zeta(x)=0$ for
$|x|\geq1$. Let 
\[
\widetilde V(x)=V+\sum_{i=1}^k\frac{\nu_i}{|x-a_{i}|^{2s}}\zeta\left(\frac{x-a_i}{\delta_i}\right).
\]
Then $\widetilde V- V\in C^\infty(\R^N\setminus\{a_1,\dots,a_k\})$ and
$\widetilde V\geq V$. Moreover $\widetilde V$ is of the form
\eqref{eq:15} with $\tilde\lambda_i=\lambda_i+\nu_i>0$ and, in view of
\eqref{eq:def_inf} and the
fact that $\widetilde V\leq \widehat V$, $\mu (\widetilde V)\geq
\mu(\widehat V)>0$. The proof is thereby complete.
\end{proof}

\section{Localization of Binding}\label{sec:localization}

This section is devoted to the proof of Theorem \ref{thm:separation},
which is the main tool needed in order to prove our main
result. Indeed this tool ensures, inside the class $\Theta$, that the
sum of two positive operators is positive, provided one of them is
translated sufficiently far.

For any $\de>0$ and $a_1,\dots,a_k\in \R^N$, we define 
 \begin{equation}\label{eq:def_P^delta}
  \mathcal{P}^\de_{a_1,\dots,a_k}:=
\mathcal{P}^\de_\infty\cap\bigg(\bigcap_{j=1}^k \mathcal{P}^\de_{a_j}\bigg)
\end{equation}
where $\mathcal{P}^\de_\infty$ is defined in \eqref{eq:def_P_infty}
and, for all $j=1,\dots,k$,
 \begin{equation*}
\begin{aligned}
\mathcal{P}^\de_{a_j}=
\bigg\{ f\colon \R^N\to\R \colon &f\in C^1(B'(a_j,R_j)\setminus\{a_j\})\text{ for some $R_j>0$}\\[-10pt]
  &\text{and}~\abs{f(x)}+\abs{(x-a_j)\cdot \nabla f(x)}=O(\abs{x-a_j}^{-2s+\de})~\text{as }x\to a_j\bigg\}.
  \end{aligned}
 \end{equation*}

\begin{lemma}\label{separation-lemma}
Let 
\begin{gather*}
 V_1(x)=\sum_{i=1}^{k_1}\frac{\la_i^1\chi_{B'(a_i^1,r_i^1)}(x)}{|x-a_i^1|^{2s}}+\frac{\la_\infty^1 \chi_{\R^N \setminus B'_{R_1}}(x)}{|x|^{2s}}+W_1(x) \in \Theta^*, \\
 V_2(x)=\sum_{i=1}^{k_2}\frac{\la_i^2\chi_{B'(a_i^2,r_i^2)}(x)}{|x-a_i^2|^{2s}}+\frac{\la_\infty^2 \chi_{\R^N \setminus B'_{R_2}}(x)}{|x|^{2s}}+W_2(x) \in \Theta^*,
\end{gather*}
with $W_1\in \mathcal{P}^\de_{a^1_1,\dots,a^1_{k_1}}$, 
$W_2\in \mathcal{P}^\de_{a^2_1,\dots,a^2_{k_2}}$ for some $\de>0$.
 If $\mu(V_1), \mu(V_2)>0$ and $\la^1_\infty
+\la_\infty^2< \ga_H$, then there exists $R>0$ such that for every $y
\in \Rn\setminus \overline{B'_R}$ there exists  $\Phi_y \in \Dext$ such
that 
$\Phi_y$ is strictly positive and locally
 H\"older continuous
 in $\overline{\R^{N+1}_+} \setminus \{(0,a_i^1),(0, a_i^2+y) \}_{i=1,
   \dots, k_j, j=1,2}$ and 
\begin{equation}\label{eq:16}
\int_{\R^{N+1}_+}t^{1-2s}\nabla \Phi_y\cdot \nabla U\dxdt\geq \kappa_s\int_{\R^N}(V_1(x)+V_2(x-y))\Tr\Phi_y\Tr U\dx
\end{equation}
for all $U\in \Dext$, with $U\geq 0$ a.e.
\end{lemma}
\begin{proof}
First of all we observe that it is not restrictive to assume that
$\lambda_i^j>0$ for all $i=1,\dots,k_j$, $j=1,2$. Indeed, letting $V_1,V_2$ as in
the assumptions, from Corollary \ref{cor:pert-at-pole} there exist
$\widetilde V_1,\widetilde V_2\in\Theta^*$ with positive masses at poles
such that $\widetilde V_j\geq V_j$ and $\mu( \widetilde V_j)>0$ for
$j=1,2$. If the theorem is true under the further assumption of
positivity of masses at poles, we conclude that,  for every $y$ with
$|y|$ sufficiently large, there exists $\Phi_y \in \Dext$ strictly positive and locally
 H\"older continuous
 in $\overline{\R^{N+1}_+} \setminus \{(0,a_i^1),(0, a_i^2+y) \}_{i=1,
   \dots, k_j, j=1,2}$ such that \eqref{eq:16} holds with $\widetilde
 V_1(x)+\widetilde V_2(x-y)$ in the right hand side integral 
instead of $V_1(x)+V_2(x-y)$. Since $\widetilde
V_1(x)+\widetilde V_2(x-y)\geq  V_1(x)+V_2(x-y)$ we obtain
\eqref{eq:16}. Then we can assume that $\lambda_i^j>0$ for all
$i=1,\dots,k_j$, $j=1,2$, without loss of generality.

Let $\eps\in(0,\gamma_H)$  be such that
$\la^1_\infty+\la^2_\infty<\ga_H-\eps$, 
 $\la^1_\infty<\ga_H-\eps$, and $\la^2_\infty<\ga_H-\eps$ and let
 $\Lambda:=\ga_H-\eps$. Let us set 
\[
\nu_\infty^1:=\Lambda-\la^1_\infty,\quad 
\nu_\infty^2:=\Lambda-\la^2_\infty,
\]
so that $\nu_\infty^1,\nu_\infty^2>0$. Let $0<\eta<1$ be such that
\begin{equation}\label{E1}
\la_\infty^2<\nu_\infty^1(1-2\eta)\quad\text{and}\quad \la_\infty^1<\nu_\infty^2(1-2\eta).
\end{equation}
Let us choose $\bar{R}>0$ large enough so that
\[
 \cup_{i=1}^{k_j} B'(a_i^j,r_i^j) \subset B'_{\bar{R}} \quad \text{for }j=1,2.
\]
We observe that, by Theorem \ref{Theorem-a}, there exists $\tilde{R}_j>0$ such that
\[
 \mu\left(V_j+\frac{\nu_\infty^j }{|x|^{2s}}\chi_{\R^N \setminus B'_{\tilde{R}_j}}\right)>0.
\]
Since $\lambda_i^j>0$ implies that $a_{\lambda_i^j}<0$, we
can 
fix some $\sigma>0$ such that $\sigma<2s$ and $\sigma<-a_{\lambda_i^j}$ for all
$i=1,\dots,k_j$, $j=1,2$.
Let us consider, for $j=1,2$, $p_j\in
C^\infty(\R^N \setminus
\{a_1^j,\dots,a_{k_j}^j\})\cap 
\mathcal{P}^\sigma_{a^j_1,\dots,a^j_{k_j}}$ such that $p_j(x)>0$ for
all $x\in\R^N$ and 
\begin{equation}\label{p-j}
p_j(x)\geq \frac{1}{|x-a_i^j|^{ 2s-\sigma}} \text{ if  } x\in
B'(a_i^j, r_i^j),\quad  p_j(x)\geq 
1  \text{ if  }x\in B'_{\bar{R}} \setminus \cup_{i=1}^{k_j}B(a_i^j, r_i^j).
\end{equation}
Since $p_j\in L^{\frac{N}{2s}}(\R^N)$ satisfies the hypotheses of Lemma \ref{lemma:compact_emb} the infimum
\[
\mu_j=\inf_{\substack{U \in \Dext \\ \Tr U\nequiv 0
  }}\frac{\int_{\R^{N+1}_+}t^{1-2s}|\na
  U|^2\dxdt-\kappa_s\int_{\Rn}\Big[V_j+\tfrac{\nu_\infty^j}{|x|^{2s}}\chi_{\R^N
      \setminus B'_{\tilde{R}_j}}\Big]\abs{\Tr
    U}^2\dx}{\int_{\R^{N}}p_j\abs{\Tr U}^2\dx}>0
\]
is achieved by some nonnegative $\Psi_j \in \Dext$, for $j=1,2$. In addition, $\Psi_j$ weakly solves
\begin{equation}\label{E2}
\left\{\begin{aligned}
-{\divergence}(t^{1-2s}\na \Psi_j) &= 0, &&\text{in } \R^{N+1}_+, \\
-\lim_{t \to 0^+}t^{1-2s}\frac{\pa \Psi_j}{\pa t}&=\kappa_s\left[V_j+\frac{\nu_\infty^j }{|x|^{2s}}\chi_{\R^N \setminus B'_{\tilde{R}_j}}+\mu_j p_j\right]\Tr\Psi_j, && \text{on } \R^{N}.
\end{aligned}\right.
\end{equation}
From Proposition \ref{prop:jlx} we know that $\Psi_j$ is
locally H\"older continuous in $\overline{\R^{N+1}_+}\setminus
\{(0,a_1^j),\dots,(0,a_{k_j}^j)\})$.

In order to prove that $\Psi_j$ is strictly positive in
$\overline{\R^{N+1}_+}\setminus
\{(0,a_1^j),\dots,(0,a_{k_j}^j)\}$, we compare
it with the unique weak solution $\widetilde\Psi_j \in \Dext$  to the problem 
\begin{equation}\label{eq:13}
\left\{\begin{aligned}
-{\divergence}(t^{1-2s}\na \widetilde\Psi_j) &=0 , &&\text{in }\R^{N+1}_+, \\
-\lim_{t \to 0^+}t^{1-2s}\frac{\pa \widetilde\Psi_j}{\pa t}&=\kappa_s
V_j\Tr \widetilde\Psi_j+\kappa_s\mu_j p_j\Tr\Psi_j, &&\text{on }\R^N,
\end{aligned}
\right.
\end{equation}
whose existence directly follows from the Lax-Milgram Lemma.
The difference $\widetilde\Phi_j=\Psi_j - \widetilde\Psi_j$ belongs to $\Dext$ and weakly solves 
\begin{equation*}
\left\{\begin{aligned}
-{\divergence}(t^{1-2s}\na \widetilde\Phi_j) &=0 , &&\text{in }\R^{N+1}_+, \\
-\lim_{t \to 0^+}t^{1-2s}\frac{\pa \widetilde\Phi_j}{\pa t}&=\kappa_s V_j\Tr \widetilde\Phi_j+\kappa_s
\frac{\nu_\infty^j }{|x|^{2s}}\chi_{\R^N \setminus B'_{\tilde{R}_j}}\Tr\Psi_j, &&\text{on }\R^N.
\end{aligned}
\right.
\end{equation*}
By testing the above equation with
$-\widetilde\Phi_j^-$ and recalling that   $\mu(V_j)>0$, we obtain that
$\widetilde\Phi_j\geq 0$ in $\R^{N+1}_+$ and hence $\Psi_j \geq \widetilde\Psi_j$.
Moreover, testing \eqref{eq:13} with
$-\widetilde\Psi_j^-$, we also obtain that $\widetilde\Psi_j\geq0$  in $\R^{N+1}_+$.
From the classical Strong Maximum
Principle and Proposition \ref{prop:cabre_sire} (whose
assumption \eqref{eq:9} for \eqref{eq:13} is satisfied thanks to Lemma \ref{l:cabre_sire} and the  assumption $V_j\in\Theta^*$) it
follows that
$\widetilde\Psi_j>0$ in
$\overline{\R^{N+1}_+}\setminus
\{(0,a_1^j),\dots,(0,a_{k_j}^j)\}$
and hence 
\[
\Psi_j>0\quad\text{in}\quad 
\overline{\R^{N+1}_+}\setminus
\{(0,a_1^j),\dots,(0,a_{k_j}^j)\}.
\]
Lemma \ref{lemma:regul_2} yields
\[
 \lim_{|x| \to \infty}\Psi_j(0,x)|x|^{N-2s+a_\Lambda}=\ell_j>0,
\]
for some $\ell_j>0$ (see \eqref{eq:6} for the
notation $a_\Lambda$).  Hence, the function $\Phi_j(t,x):=\frac{\Psi_j(t,x)}{\ell_j}$ satisfies \eqref{E2} and $\Phi_j(0,x)  \sim |x|^{-(N-2s+a_\Lambda)}$ for $\abs{x}\to\infty.$ Therefore, there exists $\rho> \max\{\tilde{R}_1,\tilde{R}_2,\bar{R} \}$ such that
\begin{equation}\label{E3}
(1-\eta^2)|x|^{-(N-2s+a_\Lambda)} \leq \Phi_j(0,x) \leq (1+\eta) |x|^{-(N-2s+a_\Lambda)}
\end{equation}
and
\begin{equation}\label{E4}
|W_1(x)| \leq \frac{\eta \nu_\infty^2}{|x|^{2s}},\qquad |W_2(x)| \leq \frac{\eta \nu_\infty^1}{|x|^{2s}}
\end{equation}
for all $x\in\R^N\setminus B'_\rho$. Also, form Lemma \ref{lemma:regul_1} we know that there exists $C>0$ such that
\begin{equation}\label{phi-j}
\frac{1}{C}|x-a_i^j|^{a_{\lambda_i^j}}  \leq \Phi_j(0,x) \leq C|x-a_i^j|^{a_{\lambda_i^j}} \quad \text{in }B'(a_i^j,r_i^j),
\end{equation}
for $i=1, \dots k_j,~j=1,2$. For any $y \in \Rn$, we define  
\[
 \Phi_y(t,x):=\nu_\infty^2\Phi_1(t,x)+\nu_\infty^1\Phi_2(t,x-y) \in \Dext.
\]
Then 
\begin{equation*}
\int_{\R^{N+1}_+}t^{1-2s}\nabla \Phi_y\cdot \nabla U\dxdt-
\kappa_s\int_{\R^N}(V_1(x)+V_2(x-y))\Tr\Phi_y\Tr U\dx
=\int_{\R^N}g_y(x)\Tr U\dx
\end{equation*}
for all $U\in \Dext$, 
where 
\begin{gather*}
 g_y(x):=\kappa_s\bigg[\mu_1 \nu_\infty^2 p_1(x)\Phi_1(0,x)
 +\frac{\nu_\infty^1 \nu_\infty^2}{|x|^{2s}}\chi_{\R^N \setminus
   B'_{\tilde R_1 }}(x)\Phi_1(0,x)
+\mu_2 \nu_\infty^1 p_2(x-y)\Phi_2(0,x-y)\\
+\frac{\nu_\infty^1 \nu_\infty^2}{|x-y|^{2s}}\chi_{\R^N \setminus
  B'(y,\tilde R_2)}(x)\Phi_2(0,x-y)-\nu_\infty^1 V_1(x)\Phi_2(0,x-y)-\nu_\infty^2 V_2(x-y)\Phi_1(0,x)\bigg].
\end{gather*}
Therefore, to conclude the proof it is enough to show that
$g_y\geq0$ a.e. in $\R^N$.

From \eqref{E1}, \eqref{E3} and \eqref{E4}, it follows that in
$\Rn \setminus \big(B'_\rho \cup B'(y,\rho)\big)$ 
\begin{align*} 
  g_y(x)&\geq \kappa_s\bigg[\frac{\nu_\infty^1
    \nu_\infty^2}{|x|^{2s}}\Phi_1(0,x)+
  \frac{\nu_\infty^1 \nu_\infty^2}{|x-y|^{2s}}\Phi_2(0,x-y)\\
 &\qquad -\nu_\infty^1 \left(\frac{\la_\infty^1}{|x|^{2s}}+W_1(x)\right)
  \Phi_2(0,x-y)-\nu_\infty^2 \left(\frac{\la_\infty^2}{|{x-y}|^{2s}}+W_2(x-y)\right)\Phi_1(0,x)\bigg]\\
&  >\kappa_s\nu_\infty^1\nu_\infty^2(1-\eta^2)\bigg[|x|^{-(N+a_\Lambda)}
  +|x-y|^{-(N+a_\Lambda)}\\
 &\qquad -|x|^{-(N-2s+a_\Lambda)}|x-y|^{-2s}-|x|^{-2s}
  |x-y|^{-(N-2s+a_\Lambda)}\bigg]\\
&=\kappa_s\nu_\infty^1\nu_\infty^2(1-\eta^2)\left(\frac1{|x|^{N-2s+a_\Lambda}}-
\frac1{|x-y|^{N-2s+a_\Lambda}}\right)
\left(\frac1{|x|^{2s}}-
\frac1{|x-y|^{2s}}\right)
\geq 0.
\end{align*}
For $|y|>R>2\rho$, we have $B'_\rho \cap B'(y,\rho)=\emptyset$. From
\eqref{p-j}, \eqref{E3}, \eqref{E4}, \eqref{phi-j}  and the choice
of $\sigma$ we have that, in $B(a_i^1,r_i^1)$,
\begin{align*}
 g_y(x)&\geq\kappa_s\bigg[\mu_1 \nu_\infty^2 p_1(x)\Phi_1(0,x)+\frac{\nu_\infty^1 \nu_\infty^2}{|x-y|^{2s}}\Phi_2(0,x-y) \\
&\qquad\qquad - \nu_\infty^1 V_1(x)\Phi_2(0,x-y)-\nu_\infty^2
  V_2(x-y)\Phi_1(0,x)\bigg]\\
&
\geq
\kappa_s|x-a_i^1|^{a_{\lambda_i^1}-2s+\sigma}\bigg[
\frac{\mu_1
  \nu_\infty^2}{C}\\
&\qquad\qquad-\nu_\infty^1(1+\eta)|x-y|^{-(N-2s+a_\Lambda)}|x-a_i^1|^{-a_{\lambda_i^1}-\sigma}(\lambda_i^1+\|W_1\|_{L^{\infty}(\R^N)}
|x-a_i^1|^{2s})\\
&\qquad\qquad-\nu_\infty^2\nu_\infty^1(1-\eta)C |x-y|^{-2s}|x-a_i^1|^{2s-\sigma}\bigg]
\\
&\geq \kappa_s|x-a_i^1|^{a_{\lambda_i^1}-2s+\sigma} \bigg[ \frac{\mu_1 \nu_\infty^2}{C} +o(1) \bigg],
\end{align*}
as $|y| \to \infty$. Now let
$x\in B'_\rho \setminus \left(\cup_{i=1}^{k_1}B'(a_i^1,r_i^1)\right)$:
since $\Phi_1$ is positive and continuous we have
$\tilde{C}^{-1}>\Phi_1(0,x)>\tilde{C}$, for some $\tilde{C}>0$, and so, thanks to
\eqref{p-j}, \eqref{E3} and \eqref{E4}, there holds
\[
 g_y(x) \geq \kappa_s\mu_1 \nu_\infty^2 \tilde C +o(1),
\]
as $|y| \to \infty$. One can similarly prove that, for $\abs{y}$ sufficiently large, $g_y(x)\geq 0$ in $B'(y,\rho)$ as well. The proof is thereby complete.
\end{proof}

\begin{proof}[Proof of Theorem \ref{thm:separation}]
	First, let
	\begin{equation}\label{L-1}
	0<\eps<\min\left\{2S \mu(V_j),\frac{\mu(V_j)}{2}\bigg[ \frac{1}{\ga_H} \left( \sum_{i=1}^k |\la_i^j| +|\la_\infty^j|\right) +S^{-1}\norm{W}_{L^{N/2s}(\Rn)} \bigg]^{-1} \right\}
	\end{equation}	
	for $j=1,2$, such that, in addition,
        $\la_\infty^1+\la_\infty^2+\eps ( \la_\infty^1 +\la_\infty^2
        )<\ga_H$ and $\lambda_i^j +\eps \lambda_i^j<\gamma_H$
         for all $i=1,\dots,k_j,\infty$. Similarly to \eqref{eq:criterion_2}, one can prove that $\mu(V_j+\eps V_j)>\frac{\mu(V_j)}{2}>0$ for $j=1,2$. Moreover, let $\sigma=\sigma(\eps)$ be such that
	\begin{equation}\label{eq:separation_1}
		0<\sigma< \min\left\{\frac{S\mu(V_1)}{2},\frac{S\mu(V_2)}{2},\frac{S\eps}{2}\right\}.
	\end{equation}
	Let, for $j=1,2$,
	\[
		 \hat{V}_j(x)=\sum_{i=1}^{k_j}\frac{\la_i^j\chi_{B'(a_i^j,r_i^j)}(x)}{|x-a_i^j|^{2s}}+\frac{\la_\infty^j \chi_{\R^N \setminus B'_{R_j}}(x)}{|x|^{2s}}+\hat{W}_j(x) \in \Theta^*,
	\]
	be such that  $\hat{W}_j\in
        \mathcal{P}^\de_{a^j_1,\dots,a^j_{k_j}}$  for some $\delta>0$ and
	\begin{equation}\label{eq:separation_2}
	\lVert\hat{V}_j-V_j\rVert_{L^{N/2s}(\R^N)}<\frac{\sigma}{1+\epsilon}.
	\end{equation}
	From Lemma \ref{lemma:approx_potentials}, \eqref{eq:separation_1} and \eqref{eq:separation_2} we deduce that
	\[
	\mu(\hat{V}_j+\eps\hat{V}_j)\geq \mu(V_j+\eps V_j)-(1+\eps)S^{-1}\lVert\hat{V}_j-V_j\rVert_{L^{N/2s}(\R^N)}>0.
	\]
	Hence we infer from Lemma \ref{separation-lemma} that there
        exists $R>0$ such that, for all
        $y\in\R^N\setminus \overline{B'_R}$, there exists
        $\Phi_y \in \Dext$   such that $\Phi_y$ is strictly positive and locally
 H\"older continuous
 in $\overline{\R^{N+1}_+} \setminus \{(0,a_i^1),(0, a_i^2+y) \}_{i=1,
   \dots, k_j, j=1,2}$ and 
	\begin{equation*}
	\int_{\R^{N+1}_+}t^{1-2s}\nabla \Phi_y\cdot \nabla U\dxdt-\kappa_s\int_{\R^N}\left[\hat{V}_1(x)+\eps\hat{V}_1(x)+\hat{V}_2(x-y)+\eps\hat{V}_2(x-y)\right]\Tr\Phi_y\Tr U\dx\geq 0
	\end{equation*}
	for all $U\in \Dext$, with $U\geq 0$ a.e. Therefore, thanks to the positivity criterion (Lemma \ref{Criterion}), we know that
	\[
		\mu(\hat{V}_1(\cdot)+\hat{V}_2(\cdot-y))\geq \frac{\eps}{\eps+1}.
	\]
Combining Lemma \ref{lemma:approx_potentials}
        with \eqref{eq:separation_1} and \eqref{eq:separation_2}, we
        finally deduce that
	\begin{multline*}
		\mu(V_1(\cdot)+V_2(\cdot-y))\geq \mu(\hat{V}_1(\cdot)+\hat{V}_2(\cdot-y)) \\
		-S^{-1}\lVert V_1-\hat{V}_1 \rVert_{L^{N/2s}(\R^N)} -S^{-1}\lVert V_2(\cdot-y)-\hat{V}_2(\cdot-y) \rVert_{L^{N/2s}(\R^N)} >0,
	\end{multline*}
		thus completing the proof.	
\end{proof}

\section{Proof of Theorem \ref{theorem}}\label{sec:thm}
In order to prove Theorem \ref{theorem}, we first need the following lemma, concerning the left-hand side in Hardy inequality \eqref{eq:hardy}.
\begin{lemma}\label{lemma:conv_hardy}
  We have that   
 \begin{equation*}
  \lim_{\abs{\xi}\to 0}\int_{\R^N}\frac{\abs{u(x)}^2}{\abs{x+\xi}^{2s}}\dx=\int_{\R^N}\frac{\abs{u(x)}^2}{\abs{x}^{2s}}\dx \quad\text{and}\quad
  \lim_{\abs{\xi}\to +\infty}\int_{\R^N}\frac{\abs{u(x)}^2}{\abs{x+\xi}^{2s}}\dx=0
\end{equation*}
for any $u\in\Ds$.
\end{lemma}
\begin{proof}
 The proof easily follows from  density of
 $C_c^\infty(\R^N\setminus\{0\})$ in $\Ds$ (see Lemma \ref{l:density}),
  the Dominated Convergence Theorem and the fractional Hardy inequality \eqref{eq:hardy}.
\end{proof}

We are now able to prove Theorem \ref{theorem}.

\begin{proof}[Proof of Theorem \ref{theorem}]
  First we prove that condition \eqref{lambda-i} is sufficient for the
  existence of at least one configuration of poles $a_1,\dots,a_k$
  such that the quadratic form associated to
  $\mathcal{L}_{\lambda_1,\dots,\lambda_k,a_1,\dots,a_k}$ is positive
  definite. In order to do this, we argue by induction on the number
  of poles $k$. For any $k$ we assume  the masses to be sorted in 
  increasing order $\lambda_1\leq\cdots\leq \lambda_k$. If $k=2$ the
  claim is proved in Remark \ref{rmk:2_poles}. Suppose now the claim
  is proved for $k-1$. If $\lambda_k\leq 0$ the proof is trivial, so
  let us assume $\lambda_k>0$: since \eqref{lambda-i} holds, it is
  true also for $\lambda_1,\dots,\lambda_{k-1}$, hence there exists a
  configuration of poles $a_1,\dots,a_{k-1}$ such that
  $Q_{\lambda_1,\dots,\lambda_{k-1},a_1,\dots,a_{k-1}}$ is positive
  definite. If we let
 \[
  V_1(x)=\sum_{i=1}^{k-1}\frac{\lambda_i}{\abs{x-a_1}^{2s}}\quad\text{and}\quad V_2(x)=\frac{\lambda_k}{\abs{x}^{2s}},
 \]
we have that $V_1,V_2\in\Theta$ satisfy the assumptions of Theorem \ref{thm:separation}. Therefore there exists $a_k\in\R^N$ such that
\[
 \mathcal{L}_{\lambda_1,\dots,\lambda_k,a_1,\dots,a_k}=(-\Delta)^s-(V_1+V_2(\cdot-a_k))
\]
is positive definite. This concludes the first part.

We now prove the necessity of condition \eqref{lambda-i}. Let $\eps>0$ be such that
\begin{equation}\label{eq:proof_thm_1}
 \norm{u}_{\Ds}^2-\sum_{i=1}^k\lambda_i\int_{\R^N}\frac{\abs{u(x)}^2}{\abs{x-a_i}^{2s}}\dx\geq \eps\norm{u}_{\Ds}^2
\end{equation}
for all $u\in\Ds$ and let $\delta\in(0,\eps\ga_H)$. Assume by
contradiction that $\lambda_j\geq \ga_H$ for some
$j\in\{1,\dots,k\}$. By optimality of $\ga_H$ in Hardy inequality
\eqref{eq:hardy} and by density of $C_c^\infty(\R^N)$ in $\Ds$, we
have that there exists $\varphi\in C_c^\infty(\R^N)$ such that
\begin{equation}\label{eq:proof_thm_2}
  \norm{\varphi}_{\Ds}^2-\lambda_j\int_{\R^N}\frac{\abs{\varphi(x)}^{2}}{\abs{x}^{2s}}\dx<\delta \int_{\R^N}\frac{\abs{\varphi(x)}^2}{\abs{x}^{2s}}\dx.
\end{equation}
If we let $\varphi_\rho:=\rho^{-\frac{N-2s}{2}}\varphi(x/\rho)$, we have that
\begin{equation}\label{eq:proof_thm_3}
 \begin{aligned}
 Q_{\lambda_1,\dots,\lambda_k,a_1,\dots,a_k}(\varphi_\rho(\cdot-a_j))&=\norm{\varphi}_{\Ds}^2-\lambda_j\int_{\R^N}\frac{\abs{\varphi(x)}^2}{\abs{x}^{2s}}\dx-\sum_{i\neq j}\lambda_i\int_{\R^N}\frac{\abs{\varphi(x)}^2}{\abs{x-\frac{a_i-a_j}{\rho}}^{2s}}\dx  \\
 & \to \norm{\varphi}_{\Ds}^2-\lambda_j\int_{\R^N}\frac{\abs{\varphi(x)}^2}{\abs{x}^{2s}}\dx \qquad\text{as }\rho\to 0^+,
\end{aligned}
\end{equation}
in view of Lemma \ref{lemma:conv_hardy}. Combining \eqref{eq:proof_thm_1}, \eqref{eq:proof_thm_2}, \eqref{eq:proof_thm_3} and Hardy inequality \eqref{eq:hardy} we obtain
\begin{equation*}
 \eps\norm{\varphi}_{\Ds}^2\leq \norm{\varphi}_{\Ds}^2-
\lambda_j\int_{\R^N}\frac{\abs{\varphi(x)}^2}{\abs{x}^{2s}}\dx \\ <\delta \int_{\R^N}\frac{\abs{\varphi(x)}^2}{\abs{x}^{2s}}\dx \leq\frac{\de}{\ga_H}\norm{\varphi}_{\Ds}^2,
\end{equation*}
which is a contradiction, because of the choice of $\delta$.

Now suppose that $K:=\sum_{i=1}^k \lambda_i\geq \ga_H$. Arguing analogously, there exists $\varphi\in C_c^\infty(\R^N)$ such that 
 \begin{equation*}\label{eq:proof_thm_4}
 \norm{\varphi}_{\Ds}^2-K\int_{\R^N}\frac{\abs{\varphi(x)}^2}{\abs{x}^{2s}}\dx<\delta \int_{\R^N}\frac{\abs{\varphi(x)}^2}{\abs{x}^{2s}}\dx.
\end{equation*}
 The function $\varphi_\rho(x):=\rho^{-\frac{N-2s}{2}}\varphi(x/\rho)$ satisfies
 \begin{equation*}
 \begin{aligned}
    Q_{\lambda_1,\dots,\lambda_k,a_1,\dots,a_k}(\varphi_\rho)&=\norm{\varphi}_{\Ds}^2-\sum_{i=1}^k\lambda_i\int_{\R^N}\frac{\abs{\varphi(x)}^2}{\abs{x-a_i/\rho}^{2s}}\dx \\
    &\to \norm{\varphi}_{\Ds}^2-K\int_{\R^N}\frac{\abs{\varphi(x)}^2}{\abs{x}^{2s}}\dx\qquad\text{as }\rho\to+\infty,
 \end{aligned}
 \end{equation*}
thanks to Lemma \ref{lemma:conv_hardy}. With the same argument as above, we again reach a contradiction.
\end{proof}

\section{Proof of Proposition \ref{Prop-1}}\label{sec:prop}
Finally, in this section we present the proof of Proposition \ref{Prop-1}, that is independent of the previous results from the point of view of the technical approach.

\begin{proof}[Proof of Proposition \ref{Prop-1}]

  First, let us denote
  $\bar{\la}=\max \{0, \la_1,\dots, \la_k, \la_\infty \}$. By
  hypothesis there exists
  $\alpha\in\big(0,1-\frac{\bar{\lambda}}{\ga_H}\big)$ such that
  $\mu(V)\leq 1-\frac{\bar{\lambda}}{\ga_H}-\alpha$. From Lemma
  \ref{Lem-2} we know that there exists $\delta>0$ such that, denoting
  by
\[
 \bar{V}=\sum_{i=1}^{k}\frac{\la_i
   \zeta(\frac{x-a_i}\delta)}{|x-a_i|^{2s}}+\frac{\la_\infty
\tilde\zeta(\frac xR)}{|x|^{2s}},
\]
with $\zeta,\tilde\zeta$ being as in Lemma  \ref{Lem-2}, 
we have that 
\begin{equation}\label{E-1}
\mu(\bar{V}) \geq 1-\frac{\bar{\la}}{\ga_H}-\frac{\al}{2}.
\end{equation}
if $\bar{\lambda}>0$ and $\mu(\bar{V})\geq 1$ if
$\bar{\lambda}=0$. We can write $V=\bar{V}+\bar{W}$ for some $\bar{W} \in L^{N/2s}(\Rn)$. Now let $\{U_n\}_n\subseteq\Dext$ be a minimizing sequence for $\mu(V)$, i.e.
\begin{equation}\label{E-2}
\int_{\R^{N+1}_+}t^{1-2s}|\na U_n|^2\dxdt-\kappa_s\int_{\Rn}V\abs{\Tr U_n}^2\dx=\mu(V)+o(1),\quad\text{as }n\to \infty
\end{equation}
and $\int_{\R^{N+1}_+}t^{1-2s}|\na U_n|^2\dxdt=1$. Since $\{U_n\}_n$ is bounded in $\Dext$, there exists $U\in\Dext$ such that, up to a subsequence (still denoted by $\{U_n\}_n$), 
\begin{equation}\label{E-3}
U_n\rightharpoonup U\quad\text{weakly in }\Dext\quad\text{and}\quad 
  U_n\to U\quad\text{a.e. in }\R^{N+1}_+,
 \end{equation}
as $n\to \infty$. There holds
\[
  \begin{aligned}
   \mu(\bar{V})& \leq \int_{\R^{N}}t^{1-2s}|\na U_n|^2\dxdt-\kappa_s\int_{\Rn}V\abs{\Tr U_n}^2\dx+\kappa_s\int_{\Rn}\bar{W}\abs{\Tr U_n}^2\dx\\
&=\mu(V)+\kappa_s\int_{\Rn}\bar{W}\abs{\Tr U_n}^2\dx+o(1), \quad\text{as } n \to \infty.
  \end{aligned}
\]
Hence, from \eqref{E-3}, \eqref{E-1}, the choice of $\alpha$, and
Lemma \ref{lemma:compact_emb} we deduce that (if $\bar{\lambda}>0$)
\[
 1-\frac{1}{\ga_H}\bar{\la}-\frac{\al}{2} \leq \mu(V) +\kappa_s\int_{\Rn}\bar{W}\abs{\Tr U}^2\dx\leq 1-\frac{\bar{\lambda}}{\ga_H}-\al+\kappa_s\int_{\Rn}\bar{W}\abs{\Tr U}^2\dx,
\]
and so
$\kappa_s\int_{\Rn}\bar{W}\abs{\Tr U}^2\dx \geq \frac{\al}{2}>0$,
which implies that $U\nequiv 0$. The same conclusion easily follows in
the case $\bar{\lambda}=0$. From the weak convergence $U_n\deb U$ in
$\Dext$, the continuity of the trace map  $\Tr: \Dext \to L^{N/2s}(\R^N)$ and the definition of $\mu(V)$,
we have that
\begin{align*}
 \mu(V)&\leq \frac{\displaystyle\int_{\R^{N+1}_+}t^{1-2s}|\na U|^2\dxdt-\kappa_s\int_{\Rn}V \abs{\Tr U}^2\dx}{\displaystyle\int_{\R^{N+1}_+}t^{1-2s}|\na U|^2\dxdt}\\
 &=\frac{\mu(V)-\bigg[\displaystyle \int_{\R^{N+1}_+}t^{1-2s}|\na
   (U_n-U)|^2\dxdt-\kappa_s\int_{\Rn}V\abs{\Tr U_n-\Tr U}^2\dx
   \bigg]+o(1)}{\displaystyle\int_{\R^{N+1}_+}t^{1-2s}|\na
   U_n|^2\dxdt-\int_{\R^{N+1}_+}t^{1-2s}|\na (U_n-U)|^2\dxdt+o(1)}\\
&\leq \mu(V)\, \frac{1-\displaystyle\int_{\R^{N+1}_+}t^{1-2s}|\na (U_n-U)|^2\dxdt+o(1)}{1-\displaystyle\int_{\R^{N+1}_+}t^{1-2s}|\na (U_n-U)|^2\dxdt+o(1)} =\mu(V)+\frac{o(1)}{\displaystyle\int_{\R^{N+1}_+}t^{1-2s}|\na U|^2\dxdt+o(1)}  .
\end{align*}
Letting $n\to \infty$ yields the fact that $\mu(V)$ is attained by $U$ and this concludes the proof.
\end{proof}

 \appendix

\section{}

In this appendix, we recall some known results about properties of
solutions to 
equations on the extended, positive half-space.

We start by recalling a regularity result.

\begin{proposition}[\cite{FF} Proposition 3, \cite{Jin2014}
  Proposition 2.6]\label{prop:jlx}
Let $a,b\in L^p(B'_1)$, for some $p>\frac{N}{2s}$ and $c,d\in L^{q}(B_1^+;t^{1-2s})$, for some $q>\frac{N+2-2s}{2}$.
 Let $w\in H^1(B^+_1;t^{1-2s})$ be a weak solution of
\begin{equation*}
\begin{cases}
-\mathrm{div}(t^{1-2s}\nabla w)+ t^{1-2s}c(z)  w=  t^{1-2s}d(z),\quad &\mbox{in } B^+_1, \\
-\lim_{t\rightarrow 0^+}t^{1-2s}\frac{\partial w}{\partial t}= a(x)w+b(x),\quad  &\mbox{on
} B'_1.
\end{cases}
\end{equation*}
Then
$w\in C^{0,\beta}(\overline{B_{1/2}^+})$ and in addition
$$
\|w\|_{  C^{0,\beta}(\overline{B_{1/2}^+})}\leq
C\left( \|w\|_{ L^2(B^+_{1})}+\|b\|_{ L^p(B_1')}+\|d\|_{L^q(B_1^+;t^{1-2s})}  \right),
$$
with $C,\beta>0$ depending only on $N,s, \|a\|_{L^{p}(B_1')},\|c\|_{L^q(B_1^+;t^{1-2s})}$.
\end{proposition}

Now we recall, from \cite{Cabre2014}, an Hopf-type Lemma.

\begin{proposition}[\cite{Cabre2014} Proposition 4.11]\label{prop:cabre_sire}
	Let $\Phi\in   C^0(B_R^+\cup B_R')\cap H^1(B_R^+;t^{1-2s})$ satisfy
	\[
		\begin{cases}
			-\divergence (t^{1-2s}\nabla \Phi)\geq 0, &\text{in } B_R^+, \\
			\Phi>0, & \text{in }B_R^+, \\
			\Phi(0,0)=0 .
		\end{cases}
	\]
	Then
	\[
		-\limsup_{t\to 0^+}t^{1-2s}\frac{ \Phi(t,0)}{ t}<0.
	\]
	In addition, if 
\begin{equation}\label{eq:9}
t^{1-2s}\frac{\partial\Phi}{\partial t}\in C^0(B_R^+\cup B_R'),
\end{equation}
then
	\[
		-\left(t^{1-2s}\frac{\partial \Phi}{\partial t}\right)(0,0)<0.
	\]
\end{proposition}

In several points of the present paper  we used the following result
from \cite{Cabre2014} to verify the validity of assumption
\eqref{eq:9} needed to apply Proposition \ref{prop:cabre_sire}.
\begin{lemma}[\cite{Cabre2014} Lemma 4.5]\label{l:cabre_sire}
  Let $s\in (0,1)$ and $R>0$. Let $\varphi\in C^{0,\sigma}(B_{2R}')$
  for some $\sigma\in(0,1)$ and $\Phi\in L^\infty(B_{2R}^+)\cap
  H^1(B_{2R}^+;t^{1-2s})$ be a weak solution to  
\[
\begin{cases}
-\divergence (t^{1-2s}\nabla \Phi)= 0, &\text{in } B_{2R}^+, \\
-\lim_{t\rightarrow 0^+}t^{1-2s}\frac{\partial \Phi}{\partial t}=
\varphi(x), &\text{on } B'_{2R}. 
\end{cases}
\]
Then there exists $\beta\in (0,1)$ depending only on $N,s,\sigma$ such
that 
\[
\Phi\in C^{0,\beta}(\overline{B_{R}^+})\quad\text{and}
\quad t^{1-2s}\frac{\partial \Phi}{\partial t}\in
C^{0,\beta}(\overline{B_{R}^+}).
\]
\end{lemma}

Finally, we prove a density result: the idea behind is that removing a
point does not impair the definition of $\Dext$ and $\Ds$;  in other
words, a point in $\R^N$ has null fractional $s$-capacity if $N>2s$,
see also \cite[Example 2.5]{AFN}. 

\begin{lemma}\label{l:density}
 Let $z_0\in\overline{\R^{N+1}_+}$, $N>2s$. Then $C_c^\infty(\overline{\R^{N+1}_+}\setminus\{z_0\})$ is dense in $\Dext$. As a consequence, if $x_0\in \R^N$, then $C_c^\infty(\R^N\setminus\{x_0\})$ is dense in $\Ds$.
\end{lemma}
\begin{proof}
 Assume $z_0\in\partial\overline{\R^{N+1}_+}=\R^N$ (the proof is completely analogous if $z_0\in\R^{N+1}_+$). Moreover, without loss of generality, we can assume $z_0=0$. Let $U\in C_c^\infty(\overline{\R^{N+1}_+})$ and let $\xi_n\in C^\infty(\overline{\R^{N+1}_+})$ be a cut-off function such that
 \begin{gather*}
  \xi_n(z)=\begin{cases}
   1, &\text{if }z\in \overline{\R^{N+1}_+\setminus B_{2/n}^+}, \\
   0, &\text{if }z\in \overline{B_{1/n}^+},
  \end{cases} \\
 \xi_n\quad\text{is radial, i.e. }\xi_n(z)=\xi_n(\abs{z}),\qquad \abs{\xi_n}\leq 1,\quad \abs{\nabla \xi_n}\leq 2n.
 \end{gather*}
Trivially $\xi_n U\in C_c^\infty(\overline{\R^{N+1}_+}\setminus\{0\})$. We claim that $\xi_n U\to U$ in $\Dext$. Indeed, thanks to Dominated Convergence Theorem,
\begin{multline*}
	\int_{\R^{N+1}_+}t^{1-2s}\abs{\nabla ((\xi_n-1) U)}^2\dxdt \\
	\leq 2\int_{\R^{N+1}_+}t^{1-2s}\abs{\xi_n-1}^2\abs{\nabla U}^2\dxdt+2\int_{\R^{N+1}_+}t^{1-2s}\abs{U}^2\abs{\nabla \xi_n}^2\dxdt \\
	\leq o(1)+Cn^2\int_{B^+_{2/n}\setminus B^+_{1/n}}t^{1-2s}\dxdt.
\end{multline*}
Moreover
\[
	n^2\int_{B^+_{2/n}\setminus B^+_{1/n}}t^{1-2s}\dxdt=O(n^{2s-N}),
\]
which concludes the proof of the claim, in view of the assumption
$N>2s$ and 
the density of $C_c^\infty(\overline{\R^{N+1}_+})$ in $\Dext$.

For what concerns the second statement, as before, without loss of
generality, we can assume $x_0=0$. Let $u\in \Ds$ and let $U\in\Dext$
be its extension. By the density of
$C_c^\infty(\overline{\R^{N+1}_+}\setminus\{0\})$  in $\Dext$ just
proved, there exists a sequence $\{U_n\}\subset
C_c^\infty(\overline{\R^{N+1}_+}\setminus\{0\})$ such that $U_n\to U$
in $\Dext$. Then $\Tr (U_n)\in C^\infty_c(\R^N\setminus\{0\})$ and
$\Tr (U_n)\to \Tr U=u$ in $\Ds$, thanks to the continuity of the trace map
 $\Tr\colon \Dext \to \Ds$.
\end{proof}

\bigbreak

\bigskip\noindent {\bf Acknowledgments.}   V. Felli is partially
supported by the PRIN2015 grant ``Variational methods, with
applications to problems in mathematical physics and geometry''.
D. Mukherjee's research is supported by the Czech Science Foundation, project GJ19--14413Y.
V. Felli and R. Ognibene are partially supported by the INDAM-GNAMPA 2018 grant ``Formula di
monotonia e applicazioni:
 problemi frazionari e stabilità spettrale rispetto a perturbazioni
 del dominio''.
This work was started while D. Mukherjee  was visiting  the University
of Milano - Bicocca supported
by INDAM-GNAMPA.

\bibliography{biblio}
\bibliographystyle{acm}

\end{document}